\documentclass[10pt]{amsart}

\usepackage{amssymb,latexsym,amsmath,amsthm}
\usepackage[mathscr]{eucal}

\usepackage{graphics,color}

\newtheorem{theorem}{Theorem}[section]
\newtheorem{lemma}[theorem]{Lemma}

\newtheorem{corollary}[theorem]{Corollary}
\theoremstyle{definition}
\newtheorem{remark}[theorem]{Remark}

\numberwithin{equation}{section}

\renewcommand{\l}{\lambda}

\newcommand{\RR}{\ensuremath{\mathbb{R}}}

\newcommand{\prtl}{\ensuremath{\partial}}

\newcommand{\supp}{\ensuremath{\text{supp}}}
\newcommand{\veps}{\ensuremath{\varepsilon}}
\newcommand{\ind}{\ensuremath{\mathbf{1}}}

\newcommand{\tubel}{\mathcal{T}_{\l^{-1/2}}(\gamma)}
\newcommand{\sqrtl}{\sqrt{\Delta_g}}
\newcommand{\Rd}{{\mathbb R}^d}
\newcommand{\dgt}{d_{\tilde g}}
\newcommand{\tg}{\tilde g}
\newcommand{\Stab}{\text{Stab}(\tilde \gamma)}
\newcommand{\sqrtg}{\sqrt{\Delta_{\tilde g}}}
\newcommand{\Sstab}{S_\la^{\text{Stab}}}
\newcommand{\Sosc}{S_\la^{\text{Osc}}}
\newcommand{\Soscg}{S_\la^{\text{Osc}, \gamma}}
\newcommand{\Sstabg}{S_\la^{\text{Stab}, \gamma}}
\newcommand{\notstab}{\Gamma\backslash \Stab}

\newcommand{\e}{\varepsilon}
\newcommand{\la}{\lambda}
\newcommand{\R}{{\mathbb R}}
\newcommand{\loc}{{\text{\rm loc}}}

\title[Kakeya-Nikodym averages in higher dimensions]{On Kakeya-Nikodym averages, $L^p$-norms
and lower bounds for nodal sets of eigenfunctions in higher dimensions}
\thanks{The first author was supported in part by the National Science Foundation grant DMS-1001529, and the second by 
the National Science Foundation grant DMS-11069175 and the Simons Foundation.}

\author[M. D. Blair]{Matthew D. Blair}
\address{Department of Mathematics and Statistics, University of New Mexico, Albuquerque, NM 87131, USA}
\email{blair@math.unm.edu}
\author[C. D. Sogge]{Christopher D. Sogge}
\address{Department of Mathematics, Johns Hopkins University, Baltimore, MD 21093, USA}
\email{sogge@jhu.edu}

\begin{document}
\begin{abstract}
We extend a result of the second author \cite[Theorem 1.1]{soggekaknik} to dimensions $d \geq 3$ which relates the size of $L^p$-norms
of eigenfunctions for $2<p<\frac{2(d+1)}{d-1}$ to the amount of $L^2$-mass in shrinking tubes about unit-length geodesics.  The proof uses
bilinear  oscillatory integral estimates of Lee~\cite{leebilinear} and a variable coefficient variant of an  ``$\veps$ removal lemma'' of   Tao and Vargas~\cite{tv1}.  We also use H\"ormander's \cite{HorOsc} $L^2$ oscillatory integral theorem and the Cartan-Hadamard theorem to show that, under the assumption of nonpositive curvature, the $L^2$-norm of eigenfunctions $e_\la$ over unit-length tubes of width $\la^{-\frac12}$ goes to zero.  Using our main estimate,
we deduce that, in this case, the $L^p$-norms of eigenfunctions for the above range of exponents is relatively small.  As a result, we can
slightly improve the known lower bounds for nodal sets in dimensions $d\ge3$  of Colding and Minicozzi~\cite{CM} in the special case of (variable) nonpositive curvature.
\end{abstract}
\maketitle

\centerline{ \bf In memoriam: {\em Lars H\"ormander (1931-2012)}}

\section{Introduction and main results}

Let $(M,g)$ be a smooth, compact boundaryless Riemannian manifold of dimension $d \geq 3$.  Let $\Delta_g$ be the nonnegative Laplace-Beltrami operator and consider eigenfunctions $e_\l$ satisfying $\Delta_g e_\l = \l^2 e_\l$ with $\l \geq 0$.
If $\varPi$ denotes the space of unit length geodesics and $dz$ the volume element associated with the metric $g$, then our main result is the following
generalizations of \cite[Theorem 1.1]{soggekaknik}:
\begin{theorem}\label{thm:mainthm}
Let $e_\l$, $\la\ge 1$,  be an eigenfunction and $\frac{2(d+2)}{d} <q <\frac{2(d+1)}{d-1}$.  Then there is a uniform constant $C<\infty$ so that
given $\e>0$ we can find a constant $C_\e$ so that
\begin{multline}\label{efcninequality}
\|e_\l\|_{L^q(M)}^q \leq \veps \l^{q(\frac{d-1}{2})(\frac 12-\frac 1q)}\|e_\l\|_{L^2(M)}^q + C\|e_\l\|_{L^2(M)}^q\\ + C_\veps \l^{q(\frac{d-1}{2})(\frac 12-\frac 1q)}\|e_\l\|^2_{L^2(M)}
\sup_{\gamma \in \varPi} \left(\int_{\mathcal{T}_{\l^{-1/2}}(\gamma)} |e_\l(z)|^2\,dz\right)^{\frac{q-2}{2}},
\end{multline}
if
\begin{equation}\label{tubedef}
\tubel=\{x\in M: d_g(x,\gamma)\le \la^{-\frac12}\}
\end{equation}
denotes the $\la^{-\frac12}$-tube about $\gamma$, with $d_g(\, \cdot \, ,\, \cdot\, )$ being the Riemannian distance function.
\end{theorem}

\begin{corollary}\label{thm:maincorollary}
The following are equivalent for any subsequence of $L^2$-normalized eigenfunctions $\{e_{\l_{j_k}}\}_{k=1}^\infty$:
\begin{align}
\limsup_{k\to \infty} \sup_{\gamma \in \varPi} \int_{\mathcal{T}_{\l_{j_k}^{-1/2}}(\gamma)} |e_{\l_{j_k}}(z)|^2\,dz &= 0 \label{kakniknonsat}\\
\limsup_{k\to \infty} \l_{j_k}^{-\frac{d-1}2(\frac12-\frac1p)}\|e_{\l_{j_k}}\|_{L^p(M)} &= 0 &\mbox{for any $\;2 < p < \frac{2(d+1)}{d-1}$} \label{lpnonsat} .
\end{align}
\end{corollary}

\begin{proof}[Proof of Corollary]
Given Theorem \ref{thm:mainthm}, it is routine to verify that \eqref{kakniknonsat} implies \eqref{lpnonsat} for $\frac{2(d+2)}{d} <p <\frac{2(d+1)}{d-1}$. The remaining values of $p$ then follow from interpolation.  For the converse, observe that H\"older's inequality gives
$$
\int_{\mathcal{T}_{\l^{-1/2}}(\gamma)} |e_{\l}(z)|^2\,dz \lesssim \l^{-\left(\frac{d-1}{2}\right)\left(1-\frac 2p\right)}\|e_{\l}\|_{L^p(M)}^2
,
$$
and the implication follows.
\end{proof}

In the case when $(M,g)$ has nonpositive sectional curvatures, we shall be able to show that \eqref{lpnonsat} holds for the full sequence of eigenvalues and hence extend the two-dimensional results of the second author and Zelditch~\cite{SZ4} to higher dimensions:

\begin{theorem}\label{negthm}
Let $(M,g)$ be a compact boundaryless manifold of dimension $d\ge 2$.  Assume further that $(M,g)$ has everywhere nonpositive sectional curvatures.
Then if $0=\la_0<\la_1\le \la_2\le \la_3\dots$ are the eigenvalues of $\sqrtl$ we have
\begin{equation}\label{i.1}
\limsup_{\la_j\to \infty}\left(\, \sup_{\gamma\in\varPi}\int_{{\mathcal{T}_{\la_j^{-1/2}}(\gamma)}} \, |e_{\la_j}|^2 \, dx \, \right)=0.
\end{equation}
Consequently, if $2<p<\frac{2(d+1)}{d-1}$, we have, in this case,
\begin{equation}\label{i.1.1}
\limsup_{\la_j\to \infty} \la_j^{-\frac{d-1}2(\frac12-\frac1p)}\|e_{\la_j}\|_{L^q(M)}=0.
\end{equation}
\end{theorem}

In \cite{sogge88} the first author showed that $\|e_\la\|_{L^q(M)}=O(\la^{\frac{d-1}2(\frac12-\frac1p)})$ when $2\le p \le\frac{2(d+1)}{d-1}$,
and that these estimates are sharp on the standard sphere $S^d$ because of the highest weight spherical harmonics.  We should point out that
for the complementary range $p>\frac{2(d+1)}{d-1}$ improved $L^p$-estimates under the above curvature assumptions follow, by interpolation
from the aforementioned $p=\frac{2(d+1)}{d-1}$ and an improved $L^\infty$-estimate which is implicit in B\'erard~\cite{Berard} (see also the second author and Zelditch \cite{soggezelduke} and \cite{soggehang}). Hassell and Tacy~\cite{HT} have recently obtained further results for  this range exponents.
Improvements for $p>\frac{2(d+1)}{d-1}$ are a bit more straightforward than \eqref{i.1.1} due to the fact that everything follows from pointwise estimates,
while, to obtain \eqref{i.1} and consequently \eqref{i.1.1}, we have to use oscillatory integrals and a finer analysis involving the deck transforms of the universal cover.  We should point out that there are no general $L^p$-improvements for the endpoint $p=\frac{2(d+1)}{d-1}$ of the results in
\cite{sogge88}, which on the sphere are saturated by eigenfunctions concentrating at points as well as ones concentrating along geodesics.

As noted before, the special case of $d=2$ of Theorem~\ref{negthm} is in \cite{SZ4}.  When $d=3$, if one assumes {\em constant} nonpositive curvature, \eqref{i.1}
follows from recent work of Chen and the second author~\cite{ChenS}, who showed that if $ds$ denotes arc length measure on $\gamma$, then
\begin{equation}\label{i.2}
\sup_{\gamma\in \varPi}\int_\gamma |e_\la|^2 ds =o(\la) \quad \text{as } \, \la \to \infty.
\end{equation}

In dimensions $d\ge 4$, Burq, G\'erard and Tzvetkov~\cite{bgtrestr}  showed that one has the following bounds for geodesic restrictions
\begin{equation}\label{i.3}
\int_\gamma |e_\la|^2 \, ds = O(\la^{d-2}).
\end{equation}
Improving this to $o(\la^{d-2})$ bounds as in \eqref{i.2} for $d=3$ is not strong enough to obtain \eqref{i.1} when $d\ge 4$.  This comes as no surprise
since, in these dimensions, \eqref{i.3} is saturated on the round sphere $S^d$ not by the highest weight spherical harmonics which concentrate along
geodesics, but rather zonal spherical harmonics, which concentrate at points.  By our main result, Theorem~\ref{thm:mainthm}, we know that \eqref{i.1}
is relevant for measuring the size of $L^p$-norms in the range $2<p<\frac{2(d+1)}{d-1}$, which are saturated on $S^d$ by highest weight
spherical harmonics.  These eigenfunctions saturate the Kakeya-Nikodym averages in \eqref{i.1}, by which we mean that the left side of \eqref{i.1}
is $\Omega(1)$, but they do not saturate the restriction estimates \eqref{i.3} for $d\ge 4$.

Fortunately, we can adopt the proof of the aforementioned improvement \eqref{i.2} of Chen and the second author~\cite{ChenS} to obtain \eqref{i.1}
in all dimensions under the assumption of nonpositive curvature.  Additionally, even for $d=3$, unlike the stronger estimate \eqref{i.2}, our techniques
do not require that we assume constant sectional curvature.

\medskip

Let us conclude this section by recording some applications of Theorems \ref{thm:mainthm} and \ref{negthm}.  First, using \eqref{i.1} we can improve the
lower bounds for $L^1$-norms of the first author and Zelditch~\cite{soggezelnod} under the above assumptions:

\begin{corollary}\label{L1}  Let $(M,g)$ be a $d$-dimensional compact boundaryless manifold with $d\ge2$.  Then
\begin{equation}\label{i.5}
\liminf_{\la\to \infty}\la^{\frac{d-1}4}\|e_\la\|_{L^1(M)}=\infty.
\end{equation}
\end{corollary}

As pointed out in \cite{soggezelnod}, no such improvement is possible for the sphere.

The proof of \eqref{i.5} is very simple.  For, by H\"older's inequality, if $p>2$,
$$1=\|e_\la\|_{L^2}\le \|e_\la\|_{L^1}^{\frac{p-2}{2(p-1)}}\, \|e_\la\|_{L^p}^{\frac{p}{2(p-1)}},$$
whence
$$\|e_\la\|_{L^p}^{-\frac{p}{p-2}}\le \|e_\la\|_{L^1}, \quad p>2.$$
As a result,
$$\bigl(\la^{-\frac{d-1}2(\frac12-\frac1p)}\|e_\la\|_{L^p}\bigl)^{-\frac{p}{p-2}}\le \la^{\frac{d-1}4}\|e_\la\|_{L^1},$$
meaning that \eqref{i.1.1} implies \eqref{i.5}.

Let us now see how \eqref{i.5}, along with an estimate of Hezari and the second author~\cite{HS} improves the known lower bounds for the
Hausdorff measure of eigenfunctions on manifolds of variable nonpositive curvature.

To this end, for a given real eigenfunction, $e_\la$, we let
$$Z_\la=\{x\in M: \, e_\la(x)=0\}$$
denote its nodal set and ${\mathcal H}^{d-1}(Z_\la)$ its $(d-1)$-dimensional Hausdorff measure.  Yau~\cite{Yau} conjectured that
${\mathcal H}^{d-1}(Z_\la)\approx \la$.  This was verified by Donnelly and Fefferman~\cite{DF} in the real analytic case and so, in particular,
if $(M,g)$ has constant sectional curvature.  The lower bound ${\mathcal H}^{d-1}(Z_\la)\ge c\la$ was verified in the $C^\infty$ case when
$d=2$ by Br\"uning~\cite{bruning} and Yau, but much less is known in this case.  An upper bound ${\mathcal H}^{d-1}(Z_\la)=O(\la^{\frac32})$ is
also known by Dong~\cite{Dong} and Donnelly and Fefferman~\cite{DF2} when $d=2$, but the best known upper bounds for $d\ge3$ are
${\mathcal H}^{d-1}(Z_\la)=O((c\la)^{(c\la)})$, which are due to Hardt and Simon~\cite{Hardt}.

Until recently, in higher dimensions for the $C^\infty$ case, the best known lower bounds for ${\mathcal H}^{d-1}(Z_\la)$ were also of an exponential
nature (see \cite{HL}).  Recently, Colding and Minicozzi~\cite{CM} and the second author and Zelditch~\cite{soggezelnod} proved lower bounds
of a polynomial nature.  Specifically, the best known lower bounds for $d\ge3$ in the $C^\infty$ case are those of Colding and Minicozzi~\cite{CM} who showed that
\begin{equation}\label{i.6}
c\la^{1-\frac{d-1}2}\le {\mathcal H}^{d-1}(Z_\la).
\end{equation}
Subsequent proofs of this using the original approach of the second author and Zeldtich~\cite{soggezelnod} were obtained by Hezari and the second author~\cite{HS} and the second author and Zelditch~\cite{soggezelnod2}.  The latter works and the earlier one \cite{soggezelnod} were based on a
variation of an identity of Dong~\cite{Dong}.

The proof of \eqref{i.6} in \cite{HS} was based on the following lower bound
\begin{equation}\label{i.7}
c\la \, \Bigl(\int_M|e_\la|\, dx\Bigr)^2\le {\mathcal H}^{d-1}(Z_\la).
\end{equation}
Indeed, simply combining \eqref{i.7} and the $L^1$-lower bound of the second author and Zelditch~\cite{soggezelnod}
\begin{equation}\label{i.8}
c\la^{-\frac{d-1}4}\le \|e_\la\|_{L^1}
\end{equation}
yields \eqref{i.6}.

Similarly, by using the improvement \eqref{i.5} of \eqref{i.8}, we can improve\footnote{An alternate approach, which yields the same sort of results, would
be to use \eqref{i.1.1} to improve the conclusion of  \cite[Lemma 4]{CM} under the assumption of nonpositive curvature.}  the known lower bounds \eqref{i.6} under our assumptions:

\begin{corollary}\label{nodalcorr}  Let $(M,g)$ be a compact boundaryless Riemannian manifold of dimension $d\ge3$ with nonpositive
sectional curvatures.  Then
\begin{equation}\label{i.9}
\liminf_{\la\to \infty}\la^{-1+\frac{d-1}2}{\mathcal H}^{d-1}(Z_\la)=\infty.
\end{equation}
In particular, when $d=3$, ${\mathcal H}^{2}(Z_\la)$ becomes arbitrarily large as $\la\to \infty$.
\end{corollary}

By a simple argument (see \cite{soggekaknik}) one always has \eqref{kakniknonsat}
and consequently \eqref{lpnonsat} as $\la$ ranges over a subsequence of eigenvalues
$\{\la_{j_k}\}$ if the resulting eigenfunctions form a quantum ergodic system (i.e. $|e_{\la_{j_l}}|^2dx$ converges in the weak${}^*$ topology to
the uniform probability measure $dx/\text{Vol}_g(M)$).  Consequently, by the above proof, we also have the following

\begin{corollary}\label{qenod}  Let $\{e_{\la_{j_k}}\}$ be a quantum ergodic system on a compact Riemannian manifold of dimension
$d\ge 3$.  We then have
\begin{equation}\label{i.10}
\lim_{k\to\infty} \la_{j_k}^{-1+\frac{d-1}2}{\mathcal H}^{d-1}(Z_{\la_{j_k}})=\infty.
\end{equation}
In particular, if the geodesic flow is ergodic, we have \eqref{i.10} as $\{\la_{j_k}\}$ ranges over a subsequence of eigenvalues of density one.
\end{corollary}

The last part of the corollary follows from the quantum ergodic theorem of Snirelman~\cite{Snirelman} / Zelditch~\cite{ZelditchQUE} / Colin de Verdi\`ere~\cite{Colin} (see also \cite{soggehang}).

\medskip

This paper is organized as follows.  In the next three sections we shall present the proof of our main result, Theorem~\ref{thm:mainthm}.  In \S 2 we shall go through the essentially routine step of reducing matters to proving certain bilinear estimates, and this step is very similar to the argument for the two-dimensional case of one of us \cite{soggekaknik}.  It gives partial control of the left side of \eqref{efcninequality} by the last term in the right.
The needed bilinear estimates, which lead to the first term in the right side of \eqref{efcninequality} are then presented in \S 3 and \S 4.  In \S 3 we show the bilinear estimate we require follows, up to an $\e$ loss, from one  of Lee~\cite[Theorem 1.1]{leebilinear}.  We then are able to remove this loss in \S 4 using a variable coefficient version of the ``$\e$-removal lemma" of Tao and Vargas~\cite[Lemma 2.4]{tv1} (see also Bourgain~\cite{bourgaincone}).  Then, in the final section, \S 5, we prove Theorem~\ref{negthm} which says that we have
 $o(1)$ bounds for $L^2$-norms over shrinking tubes under the assumption of nonpositive curvature, and consequently, by Theorem~\ref{thm:mainthm}, improved $L^p(M)$-norms for $2<p<\frac{2(d+1)}{d-1}$ of the estimates in \cite{sogge88}.

\section{Reduction to oscillatory integral estimates}
In this section, we begin the proof of Theorem \ref{thm:mainthm}, reducing matters to estimates on oscillatory integral operators.  Let $\chi_\l$ denote the operator $\chi(\sqrt{\Delta_g} - \l)$, where $\chi$ is a smooth bump function with $\chi(0)=1$ and sufficiently small compact support.  Hence $\chi_\l e_\l = e_\l$.  Recall (see Sogge, Chapter 5 \cite{soggefica}) that the kernel of this operator can be written as
$$
\chi_\l f(z) = \chi(\sqrt{\Delta_g} - \l)f(z) = \l^{\frac{d-1}{2}}\int_M e^{i\l d_g(z,y)}\alpha_\l(z,y) f(y)\,dy + R_\l f(z)
$$
where $\alpha_\l(z,y)$ is supported in $\delta \leq  d_g(z,y) \leq 2 \delta$ for some $\delta >0$ sufficiently small and less than half the injectivity radius of $(M,g)$. Moreover, $\|R_\l f\|_{L^q(M)} \lesssim \|f\|_{L^2(M)}$.

Using a sufficiently fine partition of unity, we may assume that the support of $\alpha_\l$ is sufficiently small.  In particular, we may assume that $\supp(\alpha_\l) \subset \{ |z-z_0| + |y-y_0|\ll \veps_0\}$ for some points $z_0,y_0 \in M$ with $|z_0-y_0| \approx \delta$.  Let $\gamma_0$ denote the geodesic connecting $z_0$, $y_0$ and suppose that $\Sigma$ is a suitable codimension 1 submanifold passing through $y_0$ such that $\gamma_0$ is orthogonal to $\Sigma$.   Now let $(t,s) \in \RR^{d-1}\times \RR $ denote Fermi coordinates for $\Sigma$ with $(0,0) = y_0$, $(0,s)$ parameterizing $\gamma_0$, and $(t,0)$ parameterizing $\Sigma$.  This means that for any fixed $t_0$, $(t_0,s)$ locally parameterizes the geodesic passing through $(t_0,s)$ orthogonal to $\Sigma$.

It suffices to prove that
\begin{multline*}
\int\left(\int \left|\l^{\frac{d-1}{2}}\int e^{i\l d_g(z,y)}\alpha_\l(z,(t,s)) f(t,s)\,dt\right|^2 |f(z)|^{q-2}\,dz\right)\,ds\\
\leq \veps \left(\l^{\frac{d-1}{2}(\frac 12-\frac 1q)}\|f\|_{L^2(M)} \right)^2\|f\|_{L^q(M)}^{q-2}\\ + C_\veps \l^{q(\frac{d-1}{2})(\frac 12-\frac 1q)}\|f\|^2_{L^2(M)}
\sup_{\gamma \in \varPi} \left(\int_{\mathcal{T}_{\l^{-1/2}}(\gamma)} |f(z)|^2\,dz\right)^{\frac{q-2}{2}}
\end{multline*}
Indeed, using Young's inequality for products applied to the H\"older conjugates $\frac q2$, $\frac{q}{q-2}$, we may absorb the contribution of $\veps^{(q-2)/q}\|f\|_{L^q(M)}^{q-2}$ from the first term into the left hand side, for $\veps$ sufficiently small, yielding \eqref{thm:mainthm} when $f=e_\l$.  Moreover, it suffices to prove that for each $s$ the expression in parentheses on the left hand side is bounded by the right hand side.  For convenience, we will show this for $s=0$ as the argument below works for any value of $s$ and does not use the structure of $\Sigma$ once Fermi coordinates are given.

Fix $\l$ and let $Th(z)=\int e^{i\l\psi(z,t)}\alpha_\lambda(z,(t,0)) h(t)\,dt$ where $\psi(z,t) = d_g(z,(t,0))$.  We will show that
\begin{multline}\label{hinequality}
\int |\l^{\frac{d-1}{2}}Th(z)|^2 |f(z)|^{q-2}\,dz
\leq \veps \left(\l^{\frac{d-1}{2}(\frac 12-\frac 1q)}\|h\|_{L^2_t} \right)^2\|f\|_{L^q_z}^{q-2}\\
 + C_\veps \l^{\frac{d-1}{2}}\|h\|^2_{L^2_t}
\sup_{\gamma \in \varPi} \int_{\mathcal{T}_{\l^{-1/2}}(\gamma)} |f(z)|^{q-2}\,dz.
\end{multline}
H\"older's inequality with conjugates $\frac{2}{q-2}$, $\frac{2}{4-q}$ will then imply that
$$
\l^{\frac{d-1}{2}}\int_{\mathcal{T}_{\l^{-1/2}}(\gamma)} |f(z)|^{q-2}\,dz \lesssim \l^{\frac{d-1}{2}-(\frac{d-1}{2})(2-\frac q2)}\left(\int_{\mathcal{T}_{\l^{-1/2}}(\gamma)} |f(z)|^{2}\,dz \right)^{\frac{q-2}{2}}
$$
and it is verified that the exponent of $\l$ on the right is the same as the one in \eqref{efcninequality}.

Observe that
$$
\left(Th(z)\right)^2 = \int e^{i\l(\psi(z,t)+\psi(z,t'))} \alpha_\l(z,t)\alpha_\l(z,t') h(t)h(t')\,dt\,dt'.
$$
Suppose $\veps_0$ is a small dyadic number such that $\supp(\alpha_\l(z,\cdot))\subset [-\veps_0, \veps_0]^d$ for all $z$.  Let $N>0$ be a sufficiently large dyadic number (which will essentially play the same role as the integer $N$ in \cite[(2.5)]{soggekaknik}) and let $j_0$ be the largest integer such that $2^{-j_0} \geq \l^{-\frac 12}$.  Take a Whitney-type decomposition of $[-\veps_0, \veps_0]^d \times [-\veps_0, \veps_0]^d$ away from its diagonal $D$ into almost disjoint cubes
\begin{multline*}
[-\veps_0, \veps_0]^d \times [-\veps_0, \veps_0]^d \setminus D\\
= \left( \bigcup_{\veps_0 \geq 2^j > N2^{-j_0}} \bigcup_{d(Q^j_\nu, Q^j_{\nu'}) \approx 2^{-j}} Q^j_\nu \times Q^j_{\nu'}\right)\cup \left( \bigcup_{d(Q^{j_0}_\nu, Q^{j_0}_{\nu'}) \leq  N2^{-j_0}} Q^{j_0}_\nu \times Q^{j_0}_{\nu'}\right)
\end{multline*}
where each $Q^j_\nu$ has sidelength $2^{-j}$ and is centered at a point $\nu \in 2^{-j}\mathbb{Z}^{d-1}$.  Set $h^j_\nu(t) = \mathbf{1}_{Q^j_\nu}(t)h(t)$ where the first factor denotes the indicator of the cube $Q^j_\nu$.  Hence
\begin{equation}\label{Tdecomposition}
\left(Th(z)\right)^2 = \sum_{\veps_0 \geq 2^{-j} > N2^{-j_0} } \sum_{(\nu,\nu') \in \Xi_j} Th^j_\nu(z)\,Th^j_{\nu'}(z) + \sum_{(\nu,\nu') \in \Xi_{j_0}} Th^j_\nu(z)\,Th^j_{\nu'}(z)
\end{equation}
where $\Xi_j$ denotes the collection of $(\nu,\nu')$ indexing the cubes satisfying $d(Q^j_\nu, Q^j_{\nu'}) \approx 2^{-j}$ (or $\leq N2^{-j_0}$ when $j=j_0$).

\begin{theorem}\label{thm:bilinear}
Suppose $T=T_\l$ is the oscillatory integral operator defined by
$$
T h(z) := \int e^{i\l \phi(z,s,t)} a_\l(z,s,t) h(t)\,dt
$$
where $a_\l$ is smooth and $\supp(a_\l)$ is contained in a sufficiently small uniform compact set and whose derivative bounds can be taken uniform in $\l$.  Assume further that $\phi(x,s,t)$ satisfies a Carleson-Sj\"olin type condition that $\nabla^2_{xt}\phi$ is invertible and that if $\theta$ is a unit vector for which $\nabla_{t} \langle \nabla_{(x,s)} \phi, \theta \rangle =0$, then
\begin{equation}\label{CScond}
\nabla^2_{tt} \langle \nabla_{(x,s)} \phi, \theta \rangle \text{ has eigenvalues of the same sign}.
\end{equation}
Then
\begin{equation}\label{bilinearjbound}
\left\| \sum_{(\nu,\nu') \in \Xi_j} Th^j_\nu \;Th^j_{\nu'} \right\|_{L^{q/2}_x} \lesssim 2^{j(\frac{2(d+1)}{q}-(d-1))}\l^{-\frac{2d}{q}}\|h\|_{L^2_t}^2.
\end{equation}
\end{theorem}

It can be verified that setting $z=(x,s) \in \RR^{d-1} \times \RR$, the phase function in question $\phi(x,s,t) := d_g((x,s),(t,0)) = \psi((x,s),t)$ satisfies the Carleson-Sj\"olin condition given here.  Moreover, our assumption that $q < \frac{2(d+1)}{d-1}$ ensures that the exponent of $2^j$ in \eqref{bilinearjbound} is positive.  Hence this estimate yields
\begin{equation*}
\sum_{\veps_0^{-1} \leq 2^{j} < N^{-1}2^{j_0}} \left\| \sum_{(\nu,\nu') \in \Xi_j} Th^j_\nu \;Th^j_{\nu'} \right\|_{L^{q/2}_z} \lesssim N^{-(\frac{2(d+1)}{q}-(d-1))}\l^{-\frac{d-1}{q}-\frac{d-1}{2}}\|h\|_{L^2_t}^2.
\end{equation*}
Since H\"older's inequality with conjugates $\frac q2$, $\frac{q}{q-2}$, and the triangle inequality yield
$$
\l^{d-1}\int \Big|\sum_j \sum_{(\nu,\nu') \in \Xi_j} Th^j_\nu Th^j_{\nu'} \Big|^2 |f|^{q-2}dz \leq \l^{d-1}\sum_{j} \left\| \sum_{(\nu,\nu') \in \Xi_j} Th^j_\nu Th^j_{\nu'} \right\|_{L^{q/2}_z}\|f\|_{L^q_z}^{q-2}
$$
the contribution of this sum is bounded by the first term on the right hand side of \eqref{hinequality} by taking $N$ suitably large.  This estimate can be considered as analogous to \cite[(2.6)]{soggekaknik}.

Our main tool in proving \eqref{bilinearjbound} will be a bilinear estimate due to Lee \cite[Theorem 1.1]{leebilinear} along with a refinement of arguments in \S3 of that same work.  Indeed, the estimate \eqref{bilinearjbound} should be compared with \cite[Lemma 3.3 and (3.3)]{leebilinear}.  In \cite{leebilinear}, the author proves bilinear estimates which can be thought of as a variable coefficient versions of bilinear restriction estimates due to Tao \cite{taobilinear} for elliptic surfaces (inspired by prior work of Wolff \cite{wolff} and Tao-Vargas-Vega \cite{tvv}). Lee then showed that these bilinear estimates in turn implied linear estimates on oscillatory integral operators whose phase function satisfies the Carleson-Sj\"olin type condition \eqref{CScond} (more generally called the ``H\"ormander problem").  However, his estimates suffer losses when compared to the optimal estimate predicted by scaling.  In the present work, we cannot afford such losses.  Hence one of the central tasks in this work is to prove a variable coefficient version of the $\veps$-removal lemma for bilinear estimates in \cite[Lemma 2.4]{tv1} (see also Bourgain~\cite{bourgaincone}) and refine the almost orthogonality arguments in \cite[\S3]{leebilinear}.

We now turn to the second sum in \eqref{Tdecomposition}; since $2^{-j_0} \approx \l^{-\frac 12}$ it will be treated essentially the same way as in \cite[p.527-9]{soggekaknik}. Observe that
\begin{equation*}
\left| \sum_{(\nu,\nu') \in \Xi_{j_0}} Th^j_\nu(z)\,Th^j_{\nu'}(z)\right| \lesssim N^{d-1} \sum_\nu |T h^{j_0}_\nu(z)|^2
\end{equation*}
The main estimate for this term is then
\begin{equation}\label{kaknikterm}
\int |\l^{\frac{d-1}{2}} T h^{j_0}_\nu(z)|^2 |f(z)|^{q-2}\,dz \lesssim \l^{\frac{d-1}{2}}\|h^{j_0}_\nu \|_{L^2_t}^2 \sup_{\gamma \in \varPi} \int_{\mathcal{T}_{\l^{-1/2}}(\gamma)} |f(z)|^{q-2}\,dz
\end{equation}
Since $\sum_\nu \|h^{j_0}_\nu \|_{L^2_t}^2 = \|h\|_{L^2_t}^2$, we may sum in $\nu$ to see that the contribution of these terms is bounded by the last term in \eqref{hinequality}.

To see \eqref{kaknikterm}, we will use geodesic normal coordinates centered at the point on $M$ corresponding to $(\nu,0)$ in the Fermi coordinates (recall that $\nu \in 2^{-j_0} \mathbb{Z}^{d-1}$) and let $x \mapsto \kappa(x)$ denote the diffeomorphism which makes this change of coordinates.  We may assume that $\kappa(\nu,s)=(0,s)$ (parameterizing the geodesic orthogonal to $\Sigma$ through $(\nu,0)$).   We now let $\{\omega_l\}_l$ denote a $\l^{-\frac 12}$-separated collection of points in a neighborhood of $(0,\dots,0,1)$ on $\mathbb{S}^{d-1}$ indexed by a subset of $\mathbb{Z}^{d-1}$ so that
$$
|\omega_l - \omega_k| \gtrsim \l^{-\frac 12}|l-k|.
$$
Now let
$$
S_l := \left\{ z: \left|\frac{z}{|z|} -\omega_l\right|\leq \l^{-\frac 12} \right\}
$$
and observe that the left hand side of \eqref{kaknikterm} can be dominated by
\begin{multline*}
\sum_l \left\|\l^{\frac{d-1}{2}} T(h^{j_0}_\nu) \right\|^2_{L^\infty(\kappa^{-1}(S_l))} \|f\|_{L^{q-2}(\kappa^{-1}(S_l) \cap K)}^{q-2}\\
\leq \sup_k \|f\|_{L^{q-2}(\kappa^{-1}(S_k) \cap K)}^{q-2}
\sum_l \left|\l^{\frac{d-1}{2}} T(h^{j_0}_\nu)(z_l) \right|^2
\end{multline*}
where the $z_l$ are chosen to maximize $|T(h^{j_0}_\nu)(z)|$ as $z$ ranges over $\kappa^{-1}(S_l)$ and $K$ is a small set containing the $x$-support of $\alpha_\l(x,y)$.  It thus suffices to see that for some suitable bump function $\psi$,
$$
\sum_l \left| \l^{\frac{d-1}{2}} \int e^{i\l\psi(z_l,t)} \alpha_\l(z_l,(0,t))\psi(\l^{\frac 12}(t-\nu))h^{j_0}_\nu(t)\,dt \right|^2\lesssim \l^{\frac{d-1}{2}} \|h^{j_0}_\nu\|_{L^2_t}^2.
$$
After a translation in $t$, it suffices to assume that $\nu=0$ and the desired $L^2 \to \ell^2$ estimate follows from the one dual to \eqref{l2L2dual} below.

\begin{theorem}
Suppose $\psi(z,t)$ is as defined above and $\rho(z,t)$ is a smooth bump function satisfying $|\prtl_t^\alpha \rho(z,t)| \lesssim_\alpha \lambda^{\frac{|\alpha|}{2}}$ and $\supp(\rho(\cdot,z)) \subset \{|t| \lesssim \l^{-\frac 12}\}$.  Assume also that $\rho$ vanishes when $z$ is outside of a small neighborhood $\mathcal{N}$ of $(s_0,0)$ with $s_0 \approx \delta$ with $\delta >0$ (in the Fermi coordinates described above).  Let $z_l$ be a collection of points in $\mathcal{N}$ indexed by $\mathbb{Z}^{d-1}$ such that whenever $|l-k|$ is sufficiently large,
\begin{equation}\label{kappalwr}
\left| \frac{(\kappa_1(z_l),\dots, \kappa_{d-1}(z_l))}{|\kappa(z_l)|}-\frac{(\kappa_1(z_k),\dots, \kappa_{d-1}(z_k))}{|\kappa(z_k)|}\right| \gtrsim \l^{-\frac 12} |l-k|.
\end{equation}
Then
\begin{equation}\label{l2L2dual}
\l^{\frac{d-1}{2}} \int \Big| \sum_l e^{i\l\psi(z_l,t)} \rho(z_l,t)a_l\Big|^2\,dt \lesssim \sum_l |a_l|^2.
\end{equation}
\end{theorem}

The proof of \eqref{l2L2dual} is the same as the one in \cite[Prop. 2.3]{soggekaknik}, once it is observed that
$$
|\nabla_t\psi(z_l,0)-\nabla_t\psi(z_k,0)| \gtrsim \l^{-\frac 12} |l-k|.
$$
But since the pushforward of $\prtl/\prtl z_d$ under $z\mapsto \kappa(z)$ is itself, this is a consequence of \eqref{kappalwr} and the identity
$$
\prtl_{t_i} \psi(z,0) = \langle \nu_i, \kappa(z)/|\kappa(z)| \rangle, \qquad i=1,\dots,d-1
$$
where $\nu_i$ is the pushforward of $\prtl/\prtl z_i$.

\section{Almost Orthogonality}
In this section, we begin the proof of Theorem \ref{thm:bilinear}. We first appeal to \cite[Lemma 3.1]{leebilinear} (which follows results of Bourgain~\cite{bourgainLP} and H\"ormander~\cite{HorOsc}) and the ensuing remark, which states that after a change of coordinates and multiplying $Th$, $h$ by harmless functions of modulus one, we may assume
\begin{equation}\label{phaselin}
\phi(x,s,t) = x\cdot t + \frac 12 s|t|^2 + \mathcal{E}(x,s,t)
\end{equation}
where
\begin{equation}\label{phaseerror}
\mathcal{E}(x,s,t) = O\left((|x|+|s|)^2|t|^2 \right) + O\left((|x|+|s|)|t|^3 \right).
\end{equation}

Let $\psi$ be a smooth bump function supported in $[-1,1]^{d-1}$ satisfying $\sum_{k \in \mathbb{Z}^{d-1}}\psi^2(x-k) =1$ and set $A_\mu(x) = \psi^2(2^j(x-\mu))$ with $\mu \in 2^{-j}\mathbb{Z}^{d-1}$.
\begin{lemma}\label{thm:almorthoT}
Suppose $1 \leq p \leq 2$ and that $T$ is as in Theorem~\ref{thm:bilinear}. There exist amplitudes $a_{\nu,\mu},$ $a_{\nu',\mu}$ both with $x$-support contained in $\supp(A_\mu)$ and satisfying derivative bounds of the form
\begin{equation}\label{growth2j}
|\prtl^\alpha_x a_{\nu,\mu}(x,s,t)| \lesssim_\alpha 2^{j|\alpha|}
\end{equation}
such that if $T_{\nu,\mu}$ is the oscillatory integral operator with phase $\phi$ and amplitude $a_{\nu,\mu}$
$$
T_{\nu,\mu}(h)(x,s) = \int_{\RR^{d-1}} e^{i\l\phi(x,s,t)} a_{\nu,\mu}(x,s,t)h(t)\,dt
$$
then
$$
\left\| A_\mu \sum_{\nu, \nu' \in \Xi_j } T(h^j_{\nu})T(g^j_{\nu'}) \right\|_{L^p(\RR^d)}^p \lesssim \sum_{\nu, \nu' \in \Xi_j }\left\|  T_{\nu,\mu}(h^j_{\nu})T_{\nu',\mu}(g^j_{\nu'})  \right\|_{L^p(\RR^d)}^p.
$$
\end{lemma}

\begin{proof}
For a given $s$, consider the slice of $T(h)$ at $s$ $T^s(h)(x) = T(h)(x,r)|_{r=s}$.  It suffices to show that
\begin{equation*}
\left\| A_\mu \sum_{\nu, \nu' \in \Xi_j } T^s(h^j_{\nu})T^s(g^j_{\nu'}) \right\|_{L^p(\RR^{d-1})}^p \lesssim \sum_{\nu, \nu' \in \Xi_j }\left\|  T_{\nu,\mu}^s(h^j_{\nu})T_{\nu',\mu}^s(g^j_{\nu'})  \right\|_{L^p(\RR^{d-1})}^p,
\end{equation*}
and hence we shall assume that $s$ is fixed throughout the proof.  Now let $\Phi(x,t,t') = \phi(x,s,t) + \phi(x,s,t')$ and observe that $A_\mu T^s(h^j_{\nu})T^s(g^j_{\nu'})$ can be written as
$$
A_\mu(x) \int e^{i\l \Phi(x,t,t')} a(x,s,t)a(x,s,t')h^j_{\nu}(t)g^j_{\nu'}(t')\,dtdt',
$$
Treating $D_x = -i\nabla_x$ as a vector-valued differential operator we want to write
\begin{equation}\label{bilinearderiv}
(1+(\l^{-1}2^j)^2|\l\nabla_x \Phi(\mu,\nu,\nu')-D_x|^2)^N A_\mu T^s(h^j_{\nu})T^s(g^j_{\nu'}) =T_{\nu,\mu}^s(h^j_{\nu})T_{\nu',\mu}^s(g^j_{\nu'})
\end{equation}
for some $N$ large based on $d$ and each operator on the right satisfies
\eqref{growth2j}.
It thus suffices to see that this can be done for any monomial of $$\l^{-1}2^j(\l\nabla_x \Phi(\mu,\nu,\nu')-D_x),$$ which in turn will follow by induction.  To this end, observe that products of functions satsifying \eqref{growth2j} satisfy the same condition as do weighted derivatives $(c \prtl_x)^\alpha$ of such functions provided $|c| \leq 2^{-j}$.  On $\supp(A_\mu)\times Q^j_{\nu}\times Q^j_{\nu'}$ we have that
$$
\l^{-1}2^j \left(\l\prtl_k \Phi(\mu,\nu,\nu')-\l\prtl_k \Phi(x,t,t')\right)
$$
satisfies \eqref{growth2j}.  Moreover, since $\l^{-1}2^{j} \leq 2^{-j}$, it is seen that for any $\alpha$, $(\l^{-1}2^j\prtl_x)^\alpha A_\mu^{1/2}$ satisfies \eqref{growth2j}.  The claim then follows.

It now suffices to see that if $P_{\nu,\nu'}$ is the Fourier multiplier
$$
P_{\nu,\nu'}(D_x)= (1+(\l^{-1}2^j)^2|\l\nabla_x \Phi(\mu,\nu,\nu')-D_x|^2)^{-N},
$$
then for any sequence of $\{f_{\nu,\nu'} \}$ of Schwartz class functions defined on $\RR^{d-1}$,
$$
\left\| \sum_{\nu, \nu' \in \Xi_j } P_{\nu,\nu'}  f_{\nu,\nu'} \right\|_{L^2(\RR^{d-1})}^2 \lesssim \sum_{\nu, \nu' \in \Xi_j } \| f_{\nu,\nu'} \|_{L^2(\RR^{d-1})}^2,
$$
$$
\left\| \sum_{\nu, \nu' \in \Xi_j } P_{\nu,\nu'}  f_{\nu,\nu'} \right\|_{L^1(\RR^{d-1})} \lesssim \sum_{\nu, \nu' \in \Xi_j } \| f_{\nu,\nu'} \|_{L^1(\RR^{d-1})}.
$$
The latter follows from the triangle inequality and Young's inequality for convolutions, so it suffices to treat the former.  But $\nabla_x \Phi(\mu,\nu,\nu') = 2\nabla_x \phi(\mu,s,\nu) + O(2^{-j})$, so the invertibility of $\nabla^2\phi_{x,t}$ gives
$$
2^j|\nabla_x \Phi(\mu,\nu,\nu')-\nabla_x \Phi(\mu,\tilde{\nu},\tilde{\nu}')| \approx 2^j |\nu-\tilde{\nu}|.
$$
Recall that for each $\nu$, the number of $\nu'$ such that $(\nu,\nu') \in \Xi_j$ is $O(1)$.  Therefore since the $\nu$ range over a regularly spaced $2^{-j}$ lattice, the desired bound follows from a routine computation using Plancherel's identity.
\end{proof}

Returning to the proof of Theorem \ref{thm:bilinear}, fix a pair $(\nu,\nu') \in \Xi_j$.  Set $h_1(t) = h_{\nu}^j(2^{-j}t)$, $a_{j,\nu,\mu}(x,s,t) = a_{\nu,\mu}(x,s,2^{-j}t)$, $\phi_{j}(x,s,t) = 2^{j} \phi(x,s,2^{-j}t)$ so that rescaling variables $t \mapsto 2^{-j}t$ in the integral defining $T_{\nu,\mu}(h^j_{\nu})(x,s)$ yields
$$
T_{j,\nu,\mu}(h_1)(x,s) := \int e^{i\l 2^{-j}\phi_j(x,s,t)} a_{\nu,\mu}(x,s,t)h_1(t)\,dt =  2^{j(d-1)}T_{\nu,\mu}(h^j_{\nu})(x,s).
$$
Also set $h_2(t) = h_{\nu'}^j(2^{-j}t)$ and define $T_{j,\nu',\mu}(h_2)(x,s)$ analogously, noting that $\phi_j$ remains independent of $\nu$, $\nu'$.  Moreover, we may assume that $a_{j,\nu,\mu}(x,s,\cdot)$ (resp. $a_{j,\nu',\mu}(x,s,\cdot)$) is supported in a slightly larger cube containing $\supp(h_1)$ (resp. $\supp(h_2)$).  It is helpful to observe that given \eqref{phaselin}, \eqref{phaseerror}
$$
\phi_j(x,s,t) = x\cdot t + 2^{-j-1}s|t|^2 + 2^j\mathcal{E}(x,s,2^{-j}t).
$$

\begin{lemma}\label{thm:almorthoh}
There exists an amplitude $\tilde{a}_{j,\nu,\mu}(x,s,t)$ satisfying bounds of the form \eqref{growth2j} such that
$$
(1+2^{2j}|\l^{-1} 2^{j}D_t - \mu|^2)^{N} e^{i\l2^{-j}\phi_j(x,s,t)}a_{j,\nu,\mu}(x,s,t) = e^{i\l2^{-j}\phi_j(x,s,t)}\tilde{a}_{j,\nu,\mu}(x,s,t).
$$
\end{lemma}

\begin{proof}
Observe that
$$
e^{-i\l2^{-j}\phi_j}2^j(\l^{-1}2^jD_{t_k}-\mu_k)e^{i\l2^{-j}\phi_j}a_{j,\nu,\mu} = 2^j(\prtl_{t_k}\phi_j - \mu_k)a_{j,\nu,\mu} +\l^{-1}2^{2j}D_{t_k}a_{j,\nu,\mu}.
$$
Since $\l^{-1}2^{2j} \leq 1$, second term satisfies \eqref{growth2j}.  Moreover, by \eqref{phaselin}, \eqref{phaseerror}
$$
(\prtl_{t_k}\phi_j(x,s,t) - \mu_k) = x_k-\mu_k + O(2^{-j})
$$
and thus by the support properties of $a_{j,\nu,\mu}$ the first term satisfies \eqref{growth2j} as well.  The lemma then follows by an inductive argument akin to that in Lemma \ref{thm:almorthoT}.
\end{proof}

Given this lemma we let $P_\mu=P_\mu(D_t)$ be the Fourier multiplier with symbol $P_\mu(\zeta) = (1+2^{2j}|\l^{-1} 2^{j}\zeta + \mu|^2)^{-N}$ and observe that by self-adjointness of $P_\mu(-D_t)$, we have
$$
T_{j,\nu,\mu}(h_1)(x,s) = \int e^{i\l2^{-j} \phi_j(x,s,t)} \tilde{a}_{j,\nu,\mu}(x,s,t)(P_\mu h_1)(t)\,dt
$$
Thus if we can show that
\begin{equation}\label{preparab}
\left\|T_{j,\nu,\mu}(h_1)T_{j,\nu',\mu}(h_2)\right\|_{L^{\frac{q}{2}}(\RR^d)} \lesssim \l^{-\frac{2d}{q}}2^{\frac{2j(d+1)}{q}}\|P_\mu h_1\|_{L^2(\RR^{d-1})}\|P_\mu h_2\|_{L^2(\RR^{d-1})}
\end{equation}
taking a sum with respect to $\mu$ and applying Cauchy-Schwarz will give
$$
\sum_\mu \left\|T_{j,\nu,\mu}(h_1)T_{j,\nu',\mu}(h_2)\right\|_{L^{\frac{q}{2}}(\RR^d)}^{\frac q2} \lesssim \left(\l^{-\frac{2d}{q}}2^{\frac{2j(d+1)}{q}}\right)^{\frac q2} \prod_{i=1}^2\left( \sum_\mu \|P_\mu h_i\|_{L^2(\RR^{d-1})}^q \right)^{\frac 12}
$$
and by almost orthogonality of the $P_\mu h_i$, $( \sum_\mu \|P_\mu h_i\|_{L^2}^q )^{\frac 12} \lesssim \|h_i\|_{L^2}^{\frac q2}$.  Rescaling therefore yields
\begin{equation}\label{hnubound}
\sum_\mu \left\|T_{\nu,\mu}(h^j_\nu)T_{\nu',\mu}(h^j_{\nu'})\right\|_{L^{\frac{q}{2}}(\RR^d)}^{\frac q2} \lesssim \left(\l^{-\frac{2d}{q}}2^{j(\frac{2(d+1)}{q}-(d-1))}\right)^{\frac q2} \|h^j_\nu\|_{L^2(\RR^{d-1})}^{\frac q2}\|h^j_{\nu'}\|_{L^2(\RR^{d-1})}^{\frac q2}.
\end{equation}
Hence Lemma \ref{thm:almorthoT} and Cauchy-Schwarz mean that the left hand side of \eqref{bilinearjbound} is dominated by
$$
\l^{-\frac{2d}{q}}2^{j(\frac{2(d+1)}{q}-(d-1))} \left(\sum_\nu\|h^j_\nu\|_{L^2(\RR^{d-1})}^{q}\right)^{\frac 1q}\left(\sum_\nu\|h^j_{\nu'}\|_{L^2(\RR^{d-1})}^{q}\right)^{\frac 1q}.
$$
The desired estimate \eqref{bilinearjbound} now follows from the embedding $\ell^2 \hookrightarrow \ell^q$.

We are left to show \eqref{preparab}.  At this stage, $d(\supp(h_1),\supp(h_2)) \approx 1$, but we want to exhibit the uniformity of the phases and amplitudes.  To this end, observe that
$$
\phi(x, s, t+\nu) = (x+s\nu) \cdot t + \frac 12 s|t|^2 +\mathcal{E}(x, s, t+\nu)+\frac{s}{2}|\nu|^2 + x\cdot\nu.
$$
The last two terms here can be neglected. A Taylor expansion gives
$$
\mathcal{E}(x , s, t+\nu)=\mathcal{E}(x , s, \nu)+\nabla_t \mathcal{E}(x , s, \nu)\cdot t+ \frac 12 \sum_{|\alpha| = 2}\prtl^\alpha_t \mathcal{E}(x , s, \nu)t^\alpha + R_\nu(x,s,\nu).
$$
As observed in \cite[(3.9)]{leebilinear}, we may change variables $y= x+s\nu +\nabla_t \mathcal{E}(x, s, \nu)$ and, neglecting terms which can be absorbed into either $T(h_i)$ or $h_i$, we can write
$$
\phi(y,s,t+\nu) = y\cdot t + \frac 12 s|t|^2 +\mathcal{E}_\nu(y, s, t),
$$
where $\mathcal{E}_\nu(y, s, t)$ will also satisfy \eqref{phaseerror} (with $y$ replacing $x$).  Hence
$$
\phi_j(y,s,t+2^j \nu) = 2^j\phi(y,s,2^{-j}t+\nu)=y\cdot t + 2^{-j-1} s|t|^2 +2^j\mathcal{E}_\nu(y, s, 2^{-j}t).
$$
Also define $\sigma_s = \mu+s\nu+\nabla_t\mathcal{E}(\mu,s,\nu)$ (recalling that $\mu$ is the center of the $x$-support of $\tilde{a}_{j,\nu,\mu}$, $\tilde{a}_{j,\nu',\mu}$) and observe that linearizing the change of coordinates gives that if $|x-\mu|\lesssim 2^{-j}$, then $|y-\sigma_s|\lesssim 2^{-j}$. We next set
$$
\tilde{\phi}(y,s,t) = 2^{2j}\phi_j(2^{-j}y+\sigma_s,s,t) = y\cdot t + \frac 12 s|t|^2 +2^{2j}\mathcal{E}_\nu(2^{-j}y+\sigma_s, s, 2^{-j}t +\nu)
$$
and define
$$
\widetilde{T}_1(g_1)(y,s) = \int e^{i\l2^{-2j} \tilde{\phi}(y,s,t)} \tilde{a}_{j,\nu,\mu}(2^{-j}y+\sigma_s,s,t)g_1(t)\,dt
$$
and $\widetilde{T}_2(g_2)$ in the same way except with amplitude $\tilde{a}_{j,\nu',\mu}(2^{-j}y+\sigma_s,s,t)$.  The bound \eqref{preparab} will then follow from
\begin{equation}\label{bilinparab}
\|\widetilde{T}_1(g_1)\widetilde{T}_2(g_2)\|_{L^{\frac q2}(\RR^d)} \lesssim (\l2^{-2j})^{-\frac{2d}{q}} \|g_1\|_{L^2(\RR^{d-1})}\|g_2\|_{L^2(\RR^{d-1})}.
\end{equation}
This estimate in turn follows from one of Lee \cite[Theorem 1.1]{leebilinear} along with $\veps$-removal lemmas in the next section.  We state this using his hypotheses.

For $i=1,2$, let $T_i$ be oscillatory integral operators
$$
T_i f(z) = \int e^{i\l\phi_i(z,\xi)} a_i(z,\xi) f(\xi)\,d\xi \qquad z = (x,s) \in \RR^{d-1}\times \RR, \quad \xi \in \RR^{d-1}
$$
with $a_i$ smooth and of sufficiently small compact support.  Assume that $\nabla^2_{x\xi}\phi_i$ has rank $d-1$ and that $\xi \mapsto \nabla_x \phi_i (x,s,\xi)$ is a diffeomorphism on $\supp(a_i)$.  Take $q_i(x,s,\xi) = \prtl_s \phi_i(x,s,[\nabla_x\phi_i(x,s,\cdot)]^{-1}(\xi))$ so that $\prtl_s \phi_i(x,s,\xi) = q_i (x,s,\nabla_x\phi_i(x,s,\xi))$.  Suppose further that $\nabla^2_{\xi\xi} q_i(z,\nabla_x\phi_i(z,\xi_i))$ is nonsingular for $(z,\xi_i) \in \supp(a_i)$.
\begin{theorem}\label{thm:leenoloss}
For $i=1,2$, $a_i$, $\phi_i$ satisfy the hypotheses outlined in the preceding discussion.  Set $u_i = \nabla_x \phi(z,\xi_i)$ and $\delta(z,\xi_1,\xi_2) = \nabla_\xi q_1(z,u_1) - \nabla_\xi q_2(z,u_2)$.  Then if
\begin{equation}\label{leehyp}
|\langle \nabla^2_{x\xi} \phi(z,\xi_i)\delta(z,\xi_1,\xi_2), [\nabla^2_{x\xi}\phi(z,\xi_i)]^{-1} [\nabla^2_{\xi\xi}q_i(z,u_i)]^{-1} \delta(z,\xi_1,\xi_2) \rangle| \geq c > 0
\end{equation}
for $i=1,2$, then for any $\frac{d+2}{d} <  p  $
\begin{equation}\label{bilinearnoloss}
\|T_1 f_1 T_2 f_2\|_{L^p(\RR^d)} \lesssim \l^{-\frac{d}{p}}\|f_1\|_{L^2(\RR^{d-1})}\|f_2\|_{L^2(\RR^{d-1})}.
\end{equation}
Moreover, if $T_1$, $T_2$ are members of a family of operators whose phase and amplitude functions satisfy these hypotheses uniformly and are uniformly bounded in $C^\infty$ with amplitudes supported in a set of uniform size, then the implicit constant in \eqref{bilinearnoloss} can be taken independent of each operator in the family.
\end{theorem}
We postpone the proof of this theorem until the next section.  It is then verified (see \cite[(3.14)]{leebilinear}) that if one takes $\xi = t$, $z=(x,s)$, $\tilde{\phi}(x,s,t) = \phi_1(x,s,t) = \phi_2(x,s,t)$ and $a_1$, $a_2$ as the amplitudes in $\widetilde{T}_1$, $\widetilde{T}_2$ respectively, then the left hand side of \eqref{leehyp} satisfies
$$
|\xi_1-\xi_2| + O(\veps_0) + O(2^{-j}) .
$$
Therefore since $|t_1-t_2| \approx 1$, the desired bound follows by taking \eqref{bilinearnoloss} with $\l$ replaced by $\l2^{-2j}$.

\begin{remark}
As a consequence of Theorem \ref{thm:leenoloss} and the almost orthogonality arguments in this section, we obtain the bound 
\begin{equation}\label{leelinear}
\|T h\|_{L^q(\RR^d)} \lesssim \l^{-\frac dq}\|h\|_{L^p(\RR^{d-1})} \qquad \text{when }q > \frac{2(d+2)}{d} \text{ and }\frac{d+1}{q} < \frac{d-1}{p'}
\end{equation}
for operators $T$ satisfying the hypotheses of Theorem \ref{thm:bilinear}.  In other words, we obtain Lee's estimate \cite[Theorem 1.3]{leebilinear} without the $\veps$-loss.  Indeed, the Whitney-type decomposition of $(Th)^2$ in the previous section is essentially the same as that in his work, and the estimate over the $(\nu,\nu') \in \Xi_{j_0}$ is treated on p. 85 there.  Since H\"older's inequality gives $\|h^j_\nu\|_{L^2(\RR^{d-1})} \lesssim 2^{-\frac{j(d-1)}{2} (\frac 12-\frac 1p)}\|h^j_\nu\|_{L^p(\RR^{d-1})}$, \eqref{hnubound} and the almost orthogonality arguments above yield the following variation on \eqref{bilinearjbound}
\begin{equation*}
\left\| \sum_{(\nu,\nu') \in \Xi_j} Th^j_\nu \;Th^j_{\nu'} \right\|_{L^{q/2}(\RR^d)} \lesssim 2^{j(\frac{2(d+1)}{q}-2(d-1)(1-\frac 1p))}\l^{-\frac{2d}{q}}\|h\|_{L^p(\RR^{d-1})}^2
\end{equation*}
(since it suffices to treat the cases where $q \geq p$).  Taking a sum in $j$ then yields \eqref{leelinear}.  

We also note that when $p=\infty$, the estimate in \eqref{leelinear} is valid for a larger range of $q$ by a recent work of Bourgain and Guth \cite{BG}.
\end{remark}

\section{The $\veps$-removal lemma}
Turning to the proof of \eqref{bilinearnoloss}, the estimate
\begin{equation}\label{bilinearloss}
\|T_1 f_1 T_2 f_2\|_{L^q(\RR^d)} \leq C_\alpha \l^{-\frac{d}{q}+\alpha}\|f_1\|_{L^2(\RR^{d-1})}\|f_2\|_{L^2(\RR^{d-1})}
\end{equation}
for arbitrary $\alpha >0$ and $\frac{d+2}{d}\leq q$ is due to Lee \cite[Theorem 1.1]{leebilinear}.  Moreover, as observed in \cite[p.88]{leebilinear}, the constant $C_\alpha$ is stable under small perturbations in $a_i$ and $\phi_i$.  In particular, if families of amplitudes and phase functions are considered and these functions are uniformly bounded in $C^\infty$ then $C_\alpha$ can be taken uniform within the family of operators.  The rest of this section will be dedicated to the following lemma, a generalization of \cite[Lemma 2.4]{tv1} which completes the proof of Theorem \ref{thm:leenoloss}.
\begin{lemma}
Suppose $T_1$, $T_2$ satisfy the hypotheses of the previous theorem and that they satisfy the estimate \eqref{bilinearloss} for some $1<q<\frac{d+1}{d-1}$ and some $\alpha >0$.  Assume further that
\begin{equation}\label{epsremhyp}
\frac 1p\left(1 + \frac{8\alpha}{d-1}\right) \leq \frac 1q + \frac{4\alpha}{d+1}.
\end{equation}
Then the scale-invariant estimate
$$
\|T_1 f_1 T_2 f_2\|_{L^r(\RR^d)} \lesssim \l^{-\frac{d}{r}}\|f_1\|_{L^2(\RR^{d-1})}\|f_2\|_{L^2(\RR^{d-1})}.
$$
is also valid for any $r>p$.
\end{lemma}

The hypothesis \eqref{epsremhyp} is stronger than the one appearing in \cite{tv1} (corresponding to $\sigma=\frac{d-1}2$ there)
\begin{equation}\label{tvhyp}
\frac 1p\left(1 + \frac{4\alpha}{d-1}\right) < \frac 1q + \frac{2\alpha}{d+1},
\end{equation}
but is sufficient for our purposes.

Let $f_1$, $f_2$ be unit normalized functions in $L^2(\RR^{d-1})$.  By a Marcinkiewicz interpolation argument, it suffices to see 
that\footnote{One sees that this inequality implies the lossless bilinear inequalities for each exponent $r>p$  since if
$\omega(\beta)=|\{x: \, |T_1f_1(x)T_2f_2(x)|>\beta\}|$ then $\omega(\beta)=0$ for $\beta$ larger than a fixed constant and
$\int|T_1f_1T_2f_2|^r dx = r\int_0^\infty \beta^{r-1}\omega(\beta)\, d\beta$.}
$$
\left|\left\{x: \left|T_1f_1(x)T_2f_2(x)\right| > \beta \right\}\right| \lesssim \l^{-d}\beta^{-p}.
$$
Denote the set on the left by $E$.  Observe that since $\|T_1f_1T_2f_2\|_{\infty}\lesssim 1$, it suffices to assume that $\beta \lesssim 1$.  Hence we may assume that $|E| \gtrsim \l^{-d}$ throughout since the desired bounds are guaranteed otherwise.  Moreover, we know from \eqref{bilinearloss} and Tchebychev's inequality
$$
|E| \lesssim \l^{-d+q\alpha}\beta^{-q}.
$$
Consequently it suffices to assume that $\beta > \l^{-\frac{q\alpha}{p-q}}$.  This gives the a priori bound
\begin{equation}\label{aprioriE}
|E| \lesssim \l^{-d+q\alpha(1+\frac{q}{p-q})}.
\end{equation}

Since $\beta |E| \lesssim \|\ind_E T_1f_1T_2f_2 \|_{L^1} $, it suffices to show that
$$
\|\ind_E T_1f_1T_2f_2 \|_{L^1}  \lesssim \l^{-\frac dp} |E|^{\frac{1}{p'}}.
$$
We deduce this by showing that for any unit vectors $g_1$, $g_2$ in $L^2(\RR^{d-1})$
$$
\|\ind_E T_1g_1T_2g_2 \|_{L^1}  \lesssim \l^{-\frac dp} |E|^{\frac{1}{p'}},
$$
where it should be stressed that $E$ is dependent on $f_1,f_2$ above, but that $g_1$, $g_2$ are completely independent of these functions.  Fix $g_2$ and let $T=T_{E,g_2}$ be the linear operator $Tg_1 = \ind_E T_2 g_2 T_1 g_1$.  It suffices to show that
$$
\|T^* F\|_{L^2(\RR^{d-1})} \lesssim \l^{-\frac dp}|E|^{\frac{1}{p'}}\|F\|_{L^\infty(\RR^d)},
$$
since duality then implies that $\|T g_1\|_{L^1} \lesssim \l^{-\frac dp}|E|^{\frac{1}{p'}}$.  We may assume  $\|F\|_{L^\infty} \lesssim 1$.  Set $\widetilde{F} := \ind_E T_2g_2 F$.  By a duality argument, we square the left hand side of the previous inequality to see that it suffices to show that
\begin{equation}\label{t1star}
|\langle T_1 T_1^* \widetilde{F}, \widetilde{F} \rangle| \lesssim \l^{-\frac{2d}{p}}|E|^{\frac{2}{p'}}=\l^{-2d} \, (\l^d|E|)^{\frac2{p'}},
\end{equation}
where the inner product on the left is with respect to $L^2(\RR^d)$.  The integral kernel of $T_1 T_1^*$ is
$$
K(w,z) = \int e^{i\l(\phi_1(w,\xi) - \phi_1(z,\xi))} a_1(w,\xi) \overline{a_1(z,\xi)}\,d\xi.
$$
and satisfies estimates
$$
|K(w,z)| \lesssim (1+\l|w-z|)^{-\frac{d-1}{2}}.
$$
This bound follows from the invertibility of $\nabla^2_{\xi\xi}\prtl_s \phi_1$ when $w-z$ is inside a small cone about $(0,\dots,0,1)$.  Otherwise, stronger estimates result from integration by parts and the invertibility of $\nabla_{x\xi}\phi_1$.  We now let $R\geq\l^{-1}$ be a parameter to be determined shortly and write $K(w,z) = K^R(w,z) + K_R(w,z)$ where $K^R(w,z)$ is smoothly truncated to $|w-z| \geq R$ and $K_R(w,z)$ is supported in $|w-z| \leq 2R$.  Observe that by Stein's generalization of H\"ormander's variable coefficient oscillatory integral theorem (see \cite{steinbeijing} or \cite[Ch. 2]{soggefica})
$$
\|\widetilde{F}\|_{L^1(\RR^d)} \lesssim |E|^{\frac{d+3}{2(d+1)}} \|T_2 g_2\|_{L^{\frac{2(d+1)}{d-1}}(\RR^d)}\|F\|_{L^\infty(\RR^d)}
\lesssim |E|^{\frac{d+3}{2(d+1)}} \l^{-\frac{d(d-1)}{2(d+1)}}.
$$
Thus the contribution of $K^R$ to $\langle T_1 T_1^* \widetilde{F}, \widetilde{F} \rangle$ is bounded by
$$
(\l R)^{-\frac{d-1}{2}}|E|^{\frac{d+3}{d+1}} \l^{-\frac{d(d-1)}{d+1}} =(\l R)^{-\frac{d-1}{2}} (\l^d |E|)^{\frac{d+3}{d+1}}\l^{-2d}.
$$
It is now verified that taking
$$
R= \l^{-1}\left( \l^d |E| \right)^{\frac{2}{d-1}(\frac{d+3}{d+1}-\frac{2}{p'})} \geq \l^{-1},
$$
ensures that the contribution of $K^R$ is acceptable towards proving \eqref{t1star} (by scaling, this is consistent with the choice of $R$ in \cite[Lemma 2.4]{tv1}).  We also remark that another computation reveals that \eqref{aprioriE} along with the hypothesis \eqref{epsremhyp} ensures that $R \lesssim \l^{-\frac 12}$.

It remains to control the contribution of $K_R$ to (4.5). Let $\{\psi_k\}_k$ be a partition of unity over $[-\veps_0,\veps_0]^d$ such that $\supp(\psi_k)$ is contained in a cube of sidelength $2R$ centered at a point $w_k \in R\mathbb{Z}^d$.  Let $P_R$ be the operator determined by the integral kernel $K_R$ and observe that its contribution to the left hand side of \eqref{t1star} is dominated by
\begin{equation}\label{PRlocal}
\sum_{k,k'} |\langle P_R(\psi_k \widetilde{F}), \psi_{k'}\widetilde{F} \rangle |.
\end{equation}
Given a fixed $k$, the number of $k'$ for which $\langle P_R(\psi_k \widetilde{F}), \psi_{k'}\widetilde{F} \rangle \neq 0$ is $O(1)$ and satisfies $d(\supp(\psi_k),\supp(\psi_{k'})) \lesssim R$.  Hence we will restrict attention to the sum over the diagonal $k=k'$, as a slight adjustment of the argument below will handle the off-diagonal terms.

At this stage it will be convenient to use the semiclassical Fourier transform with $h = \frac{1}{\l}$ (cf. \cite[\S3.3]{zworski})
\begin{equation}\label{semiclassft}
\mathscr{F}_{1/\l}(G)(\eta) = \int_{\RR^d} e^{-i\l w\cdot \eta } G(w)\,dw, \qquad \mathscr{F}_{1/\l}^{-1}(g)(w) =\frac{\l^d}{(2\pi)^d}\int_{\RR^d} e^{i\l w\cdot \eta } G(\eta)\,d\eta.
\end{equation}
Since $\mathscr{F}_{1/\l}$ is related to the usual Fourier transform by $\mathscr{F}_{1/\l}(G)(\eta) = \mathscr{F}(G)(\lambda \eta)$, we have the Plancherel identity $(2\pi)^d \|G\|_{L^2}^2 = \l^d \|\mathscr{F}_{1/\l}(G)\|_{L^2}^2$ (cf. \cite[(3.3.6)]{zworski}).  We now have
\begin{equation}\label{PRlocalterms}
(2\pi)^d\langle P_R(\psi_k \widetilde{F}), \psi_{k}\widetilde{F} \rangle = \l^d\left\langle \mathscr{F}_{1/\l}\big(P_R(\psi_k \widetilde{F})\big), \mathscr{F}_{1/\l}\big(\psi_k \widetilde{F}\big) \right\rangle,
\end{equation}
and the right hand side can be written as
$$
\frac{1}{(2\pi)^d}\iint J_k(\eta,\zeta) \mathscr{F}_{1/\l}\big(\psi_k \widetilde{F}\big)(\zeta) \, \overline{\mathscr{F}_{1/\l}\big(\psi_k \widetilde{F}\big)(\eta)}\,d\zeta d\eta,
$$
where
$$
J_k(\eta,\zeta) = \l^{2d}\iiint e^{-i\l(\eta\cdot w - \phi_1(w,\xi) + \phi_1(z,\xi) - z\cdot \eta)} \widetilde{\psi}_k(z,w) a_1(w,\xi)\overline{a_1(z,\xi)}\,dz\,dw\,d\xi,
$$
for some $\widetilde{\psi}_k$ supported in $|z-w_k|,|w-w_k| \lesssim R$ satisfying $|\prtl^\alpha_{w,z}\widetilde{\psi}_k|\lesssim_\alpha R^{-|\alpha|}$.  Strictly speaking, one needs to justify the use of Fubini's theorem here, but this can be done by passing to Schwartz class approximations to $\widetilde{F}$ and employing crude $L^2$ continuity bounds for $P_R$.  Therefore over $\supp(\widetilde{\psi}_k)$,
$$
|\nabla_w\phi_1(w_k,\xi)-\nabla_w\phi_1(w,\xi)| + |\nabla_z\phi_1(w_k,\xi)-\nabla_z\phi_1(z,\xi)|\lesssim R \lesssim (\l R)^{-1},
$$
where we use that $R \lesssim \l^{-\frac 12}$ in the second inequality.  Hence integration by parts gives for any $N$ and some uniform cube $Q \subset \RR^{d-1}$
$$
|J_k(\eta,\zeta)| \lesssim_N (\l R)^{2d}\int_Q (1 + \l R|\eta-\nabla_w \phi_1(w_k,\xi)|+\l R|\zeta-\nabla_w \phi_1(w_k,\xi)| )^{-N}\,d\xi,
$$
as the domain of integration in $(w,z)$ is of volume $R^{2d}$.  Let $S^1_k$ denote the hypersurface $\{\nabla \phi_1(w_k,\xi): \xi \in Q \}$. This in turn allows us to deduce that
$$
|J_k(\eta,\zeta)| \lesssim_N (\l R)^{d+1} (1 + \l R\, d(\eta,S^1_k)+\l R \,d(\zeta,S^1_k) +\l R|\zeta-\eta| )^{-N}.
$$
Consequently, by using Cauchy-Schwarz in \eqref{PRlocalterms} we have that
$$
\left|\langle P_R(\psi_k \widetilde{F}), \psi_{k}\widetilde{F} \rangle \right| \lesssim_N \l R \int (1+\l R|d(\eta,S^1_{k})|)^{-N}|\mathscr{F}_{1/\l}\big(\psi_k \widetilde{F}\big)(\eta)|^2\,d\eta.
$$
Now let $S^1_{k,l}$ denote the $(\l R)^{-1} 2^l$ neighborhood of $S^1_k$.  We have that
$$
\sum_{k} |\langle P_R(\psi_k \widetilde{F}), \psi_{k}\widetilde{F} \rangle | \lesssim_N \sum_k \sum_{l=0}^\infty \l R 2^{-lN}\|\mathscr{F}_{1/\l}\big(\psi_k \widetilde{F}\big)\|_{L^2(S^1_{k,l})}^2.
$$
We examine the case $l=0$, the other cases are similar and aided by the factor of $2^{-lN}$.  Let $\{\tilde{g}_{1,k}\}_k$ be a sequence of functions with $\supp(\tilde{g}_{1,k})\subset S^1_{k,0}$ for each $k$ and $\sum_k \|\tilde{g}_{1,k}\|_{L^2(S^2_{k,0})}^2=1$.  To finish the proof and show that \eqref{PRlocalterms} is dominated by the right side of \eqref{t1star}, it suffices to see that
$$
\sum_k \left|\langle\psi_k \widetilde{F}, \l^{-d}\mathscr{F}_{1/\l}^{-1}(\tilde{g}_{1,k})\rangle\right| \lesssim \l^{-\frac{d}{p}} (\l R)^{-\frac 12}|E|^{\frac{1}{p'}}\|F\|_{L^\infty(\RR^d)},
$$
which in turn follows from
$$
\sum_k \int \left|\psi_k \ind_E \l^{-d}\mathscr{F}_{1/\l}^{-1}(\tilde{g}_{1,k})T_2g_2\right| \lesssim \l^{-\frac{d}{p}} (\l R)^{-\frac 12}|E|^{\frac{1}{p'}}.
$$
Now reverse the roles of $g_1$ and $g_2$ from the previous step, treating $\{\tilde{g}_{1,k}\}_k$ as a fixed sequence and redefine $T = T_{E,\{g_{1,k}\}}$ by $Tg_2 = \{\psi_k \ind_E \mathscr{F}_{1/\l}^{-1}(\tilde{g}_{1,k})T_2g_2 \}_k$ so that it suffices to show
$$
\|T\|_{L^2(\RR^{d-1}) \to l^1_k L^1(\RR^d)} \lesssim \l^{d-\frac{d}{p}} (\l R)^{-\frac 12}|E|^{\frac{1}{p'}}.
$$
Let $\{F_k\}_k$ be any sequence of functions satisfying $\sup_k\|F_k\|_{L^\infty(\RR^d)} \leq 1$ and set $\widetilde{F}_k = \psi_k \ind_E \l^{-d}\mathscr{F}_{1/\l}^{-1}(\tilde{g}_{1,k})F_k$. By duality, the desired bound on $T$ will follow from
$$
\left\|T^*\left(\sum_k F_k\right)\right\|_{L^2(\RR^{d-1})} \lesssim \l^{-\frac{d}{p}} (\l R)^{-\frac 12}|E|^{\frac{1}{p'}},
$$
or equivalently
$$
\sum_{k,k'}\langle T_2 T_2^*(\widetilde{F}_k), \widetilde{F}_{k'}\rangle \lesssim \left(\l^{-\frac{d}{p}} (\l R)^{-\frac 12}|E|^{\frac{1}{p'}}\right)^2 .
$$
Observe that
\begin{multline}\label{stprep}
\sum_k \|\widetilde{F}_k\|_{L^1(\RR^d)} \leq \sum_k \|F_k\|_{L^{\infty}(\RR^d)} \int |\psi_k \ind_E \l^{-d}\mathscr{F}_{1/\l}^{-1}(\tilde{g}_{1,k})|\,dw\\
\leq \left( \int_E \sum_k \psi_k^{\frac{2(d+1)}{d+3}}\right)^{\frac{d+3}{2(d+1)}} \|\l^{-d}\mathscr{F}_{1/\l}^{-1}(\tilde{g}_{1,k})\|_{\ell^{\frac{2(d+1)}{d-1}}_k L^{\frac{2(d+1)}{d-1}}(\RR^d)}.
\end{multline}
By finite overlap of the $\supp(\psi_k)$, the first factor on the right is seen to be bounded by $|E|^{\frac{d+3}{2(d+1)}}$. Similar to before, an application of the Stein-Tomas theorem for $S^1_k$ gives
$$
\|\l^{-d}\mathscr{F}_{1/\l}^{-1}(\tilde{g}_{1,k})\|_{\ell^{\frac{2(d+1)}{d-1}}_k L^{\frac{2(d+1)}{d-1}}(\RR^d)} \lesssim \l^{-\frac{d(d-1)}{2(d+1)}} (\l R)^{-\frac 12} \|\tilde{g}_{1,k}\|_{\ell^2_k L^{2}(S^1_{k,0})},
$$
(cf. the formula for $\mathscr{F}_{1/\l}^{-1}$ in \eqref{semiclassft}) where we use that $\ell^{2}_k \hookrightarrow \ell^{\frac{2(d+1)}{d-1}}_k$.  Decomposing the integral kernel of $T_2 T_2^*$ as a sum $K^R+K_R$ as before, we may handle the contribution of $K^R$ by using \eqref{stprep} to reason analogously to the argument above.  We are thus reduced to handling the contribution of $K_R$ and denote the corresponding operator as $P_R$. As before, we restrict attention to the diagonal terms, and are thus reduced to seeing that
$$
\sum_k \lambda^d \left|\langle \mathscr{F}_{1/\l}(P_R(\widetilde{F}_k)),\mathscr{F}_{1/\l}(\widetilde{F}_k) \rangle\right| \lesssim  \left(\l^{-\frac{d}{p}} (\l R)^{-\frac 12}|E|^{\frac{1}{p'}}\right)^2.
$$
Since $\supp(\widetilde{F}_k) \subset \supp(\psi_k)$, this analogously reduces to showing that
$$
\sum_k \sum_{l=0}^\infty 2^{-lN} \|\mathscr{F}_{1/\l}(\widetilde{F}_k) \|^2_{L^2(S^2_{k,l})}\lesssim  \left(\l^{-\frac{d}{p}} (\l R)^{-1}|E|^{\frac{1}{p'}}\right)^2,
$$
where this time $S^2_{k,l}$ denotes the $(\l R)^{-1}2^l$ neighborhood of the hypersurface $S^2_k=\{\nabla \phi_2(w_k,\xi): \xi \in Q \}$.  We again restrict attention to the $l=0$ case, and let $\{\tilde{g}_{2,k}\}_k$ be a sequence such that $\supp(\tilde{g}_{2,k}) \subset S^2_{k,0}$ and $\sum_k \|\tilde{g}_{2,k}\|_{L^2(S^2_{k,0})}^2 =1$.  Observe that
$$
\sum_k \left|\langle\widetilde{F}_k, \l^{-d}\mathscr{F}_{1/\l}^{-1}(\tilde{g}_{2,k})\rangle\right| \lesssim \sum_k \|F_k\|_{L^\infty(\RR^d)}\int |\l^{-2d}\psi_k\ind_E\mathscr{F}_{1/\l}^{-1}(\tilde{g}_{1,k})\mathscr{F}_{1/\l}^{-1}(\tilde{g}_{2,k})|\,dw,
$$
and it suffices to show that the right hand side is bounded by $\l^{-\frac{d}{p}} (\l R)^{-1}|E|^{\frac{1}{p'}}$. But each term on the right is bounded by
$$
|E|^{\frac{1}{q'}}
\left(\int |\l^{-2d}\psi_k\mathscr{F}_{1/\l}^{-1}(\tilde{g}_{1,k})\mathscr{F}_{1/\l}^{-1}(\tilde{g}_{2,k})|^q\,dw \right)^{\frac 1q}.
$$
Rescaling $w \mapsto R w$ and applying the bilinear estimates \eqref{bilinearloss} (or even those in \cite{taobilinear}) shows the preceding term is bounded by
$$
|E|^{\frac{1}{q'}}R^{\frac dq}(\l R)^{-\frac{d}{q}+\alpha} (\l R)^{-1}\|\tilde{g}_{1,k}\|_{L^2(S^1_{k,0})}\|\tilde{g}_{2,k}\|_{L^2(S^2_{k,0})}.
$$
Taking the sum in $k$ and applying Cauchy-Schwarz completes the proof once we observe that 
$$
|E|^{\frac{1}{q'}}\l^{-\frac{d}{q}+\alpha}R^\alpha \lesssim |E|^{\frac{1}{p'}} \l^{-\frac{d}{p}} .
$$
Recalling that $R\approx \l^{-1}\left( \l^d |E| \right)^{\frac{2}{d-1}(\frac{d+3}{d+1}-\frac{2}{p'})} = \l^{-1}\left( \l^d |E| \right)^{\frac{2}{d-1}(\frac 2p-\frac{d-1}{d+1})} $ this inequality is equivalent to
$$
|E|^{-\frac{1}{q}}\l^{-\frac{d}{q}}\left( \l^d |E| \right)^{\frac{2\alpha}{d-1}(\frac 2p-\frac{d-1}{d+1})} \lesssim |E|^{-\frac{1}{p}} \l^{-\frac{d}{p}},
$$
which in turn can be rearranged as
$$
\left( \l^d |E| \right)^{\frac{2\alpha}{d-1}(\frac 2p-\frac{d-1}{d+1})} \lesssim \left( \l^d |E| \right)^{\frac 1q-\frac 1p} .
$$
But since $\l^{d}|E| \gtrsim 1$, this follows once it is observed that \eqref{tvhyp} is equivalent to 
$$
\frac{2\alpha}{d-1}\left(\frac 2p-\frac{d-1}{d+1}\right) < \frac 1q-\frac 1p.
$$
Even though we only need the weaker condition \eqref{tvhyp} to conclude the argument, the stronger hypothesis \eqref{epsremhyp} is used above in a significant way to ensure that $\lambda R^2 \leq 1$.

\section{$L^2$ estimates for shrinking tubes}
By Corollary~\ref{thm:maincorollary}, \eqref{i.1.1} follows from \eqref{i.1}.
Therefore, if , as before, $\varPi$ denotes the space of unit length geodesics, we must show that if $(M,g)$ has nonpositive sectional curvature, then if $\e>0$ is fixed there is a $\Lambda_\e<\infty$ so that
\begin{equation}\label{p.1}
\int_{\tubel}|e_\la|^2 \, dx\le \e, \quad \la\ge \Lambda_\e,  \, \, \, \gamma\in \varPi.
\end{equation}
Here, as before, we are denoting the volume element associated to the metric simply by $dx$.

We shall first fix $\gamma\in \varPi$ and prove the special case
\begin{equation}\label{p.2}
\int_{\tubel} |e_\la|^2 \, dx \le \e, \quad \la \ge \Lambda_\e.
\end{equation}
After doing this we shall see that we can adapt its proof using the compactness of $\varPi$ to obtain the estimates \eqref{p.1} which are uniform as $\gamma$ ranges over this space.

To prove these estimates, we shall want to use a reproducing operator which is similar to the local one, $\chi_\la$, that was used to prove
Theorem~\ref{thm:mainthm}.  This operator was a local one, but to able to take advantage of our curvature assumptions and make use of the method
of time-averaging, it will be convenient to use a variant that is in effect scaled in the spectral parameter.

To this end, let us fix a real-valued function $\rho\in {\mathcal S}(\R)$ satisfying
\begin{equation}\label{p.3}
\rho(0)=1, \quad \Hat \rho(t)=0 \, \, \text{if } \, \, |t|\ge \frac14 \, 
\, \, \text{and } \, \, \Hat \rho(t)=\Hat \rho(-t).
\end{equation}
Then for a given fixed $T\gg 1$ we have
$$\rho(T(\la -\sqrtl))e_\la = e_\la.$$
As a result, we would have \eqref{p.2} if we could show that there is a uniform constant $A=A(M,g)$ so that whenever $T\gg 1$ is fixed there is a constant
$A_T<\infty$ so that for $\la\ge 1$ we have
\begin{equation}\label{p.4}
\bigl\| \rho(T(\la-\sqrtl))f\bigl\|_{L^2(\tubel)}\le \Bigl(AT^{-\frac14}+A_T\la^{-\frac18}\Bigr)\, \|f\|_{L^2(M)}.
\end{equation}
Since $\rho(T(\la-\sqrtl)): L^2(M)\to L^2(M)$ is self-adjoint, by duality, \eqref{p.4} is equivalent to the following
\begin{equation}\label{p.5}
\bigl\| \rho(T(\la-\sqrtl))h\bigl\|_{L^2(M)}\le \Bigl(AT^{-\frac14}+A_T\la^{-\frac18}\Bigr)\, \|h\|_{L^2(M)},
\quad \text{if  supp }h\subset \tubel.
\end{equation}
If we now let
\begin{equation}\label{p.6}
m(\tau)=\bigl(\rho(\tau)\bigr)^2,
\end{equation}
we can square the right side of \eqref{p.5} to see that whenever $h$ is supported in the tube $\tubel$ we have
\begin{multline*}\bigl\|\rho(T(\la-\sqrtl))h\bigr\|^2_{L^2(M)}=\bigl\langle m(T(\la-\sqrtl))h, h\bigr\rangle
\\
\le \bigl\|m(T(\la-\sqrtl))h\|_{L^2(\tubel)}\|h\|_{L^2(\tubel)}.
\end{multline*}
Whence we deduce that our desired inequalities \eqref{p.2}, \eqref{p.4} and \eqref{p.5} would all follow from
\begin{multline}\label{p.7}
\bigl\| m(T(\la-\sqrtl))h\bigl\|_{L^2(\tubel)}
\\
\le \Bigl(CT^{-\frac12}+C_T\la^{-\frac14}\Bigr)\, \|h\|_{L^2(M)},
\quad \text{if  supp }h\subset \tubel,
\end{multline}
with $C$ and $C_T$ being equal to $A^2$ and $A_T^2$, respectively.

Since, by \eqref{p.3},
$$\Hat m(\tau)=(2\pi)^{-1}\bigl(\Hat \rho * \Hat \rho\bigr)(\tau)$$
is supported in $|\tau|<1$, we can write
$$m(T(\la-\sqrtl))=\frac1{2\pi T}\int_{-T}^T \Hat m(\tau/T) e^{i\la \tau}e^{-i\tau \sqrtl}\, d\tau.$$
After perhaps multiplying the metric, we may assume that the injectivity radius of the manifold is larger than $10$.  Let us then fix an even bump function
$\beta \in C^\infty_0(\R)$ satisfying
$$\beta(\tau)=1, \, \,  |\tau|\le \frac32, \quad \text{and } \, \, \beta(\tau)=0, \, \, |\tau|\ge 2.$$
We then can split
$$m(T(\la-\sqrtl))=R_\la + W_\la$$
where (suppressing the $T$-dependence)
$$W_\la =\frac1{2\pi T}\int_{-T}^T (1-\beta(\tau)) \Hat m(\tau/T) e^{i\la \tau}e^{-i\tau \sqrtl}\, d\tau,$$
and, if $r_T(\tau)$ denotes the inverse Fourier transform of $\tau \to \beta(\tau)\Hat m(\tau/T)$,
$$R_\la h= T^{-1}r_T(\la-\sqrtl)h.$$
Clearly, $|r_T(\tau)|\le B$ for some $B$ independent of $T\ge 1$, and therefore,
$$\|R_\la f\|_{L^2(M)}\le BT^{-1}\|f\|_{L^2(M)}, \quad T\ge 1.$$
As a result, we would obtain \eqref{p.7} if we could show that
\begin{equation}\label{p.8}
\|W_\la h\|_{L^2(\tubel)}\le \Bigl(CT^{-\frac12}+C_T\la^{-\frac14}\Bigr)\|h\|_{L^2}, \quad \text{if  supp }h\subset \tubel.
\end{equation}
By Euler's formula, if $\tilde m_T$ denotes the inverse Fourier transform of $T^{-1}(1-\beta(\tau)) \, \Hat m(\tau/T)$, we have
$$W_\la =\frac1{2\pi T}\int_{-T}^T (1-\beta(\tau)) \Hat m(\tau/T) e^{i\la\tau}\cos(\tau\sqrtl)\, d\tau + \tilde m_T(\la+\sqrtl).$$
Since $\tilde m_T(\la +\sqrtl)$ has a kernel which, for $T\ge1$, is $O_T((1+\la)^{-N})$ for every $N=1,2,\dots$ as $\tilde m_T\in {\mathcal S}(\R)$,
we conclude that we would obtain \eqref{p.8} if we could prove that
\begin{equation}\label{p.9}
\|S_\la h\|_{L^2(\tubel)}\le \Bigl(CT^{-\frac12}+C_T\la^{-\frac14}\Bigr)\|h\|_{L^2}, \quad \text{if  supp }h\subset \tubel,
\end{equation}
with
\begin{equation}\label{p.10}
S_\la =\frac1{2\pi T}\int_{-T}^T (1-\beta(\tau))\Hat m(\tau/T) e^{i\la \tau} \cos(\tau\sqrtl) \, d\tau.
\end{equation}

It is at this point that we shall finally use our hypothesis that $(M,g)$ has nonpositive sectional curvature.  By the Cartan-Hadamard theorem
(see \cite{Chavel}, \cite{do Carmo}), for each point $P\in M$, the exponential map at $P$, $\exp_P$, sending $T_PM$, the tangent space
at $P$, to $M$ is a universal covering map.  For our sake, it is natural to take $P$ to be the center of our unit-length geodesic segment
$\gamma$.  Thus, with this choice, if we identify $\Rd$ with $T_PM$, we have that
\begin{equation}\label{p.11}
\kappa=\exp_P: \Rd \simeq T_PM \to M
\end{equation}
is a covering map.

If $\tilde g=\kappa^*g$ denotes the pullback via $\kappa$ of the metric $g$ to $\Rd$, it follows that $\kappa$ is a local isometry.  We let
$\dgt(y,z)$ denote the Riemannian distance with respect to $\tg$ of $y,z\in \Rd$.  By the Cartan-Hadamard theorem there are no conjugate
points for either $(M,g)$ or $(\Rd, \tg)$.
Also, the image under $\kappa$ of any geodesic in $(\Rd, \tilde g)$ is one in $(M,g)$.  If $\{\gamma(t): \, t\in \R\}$ denotes the parameterization by
arc length of the extension of our geodesic segment $\gamma\in \varPi$, let $\tilde \gamma =\{\tilde \gamma(t): \, t\in \R\}$ denote the lift of this extension, which is the unique geodesic in $(\Rd, \tilde g)$ that passes through the origin and satisfies $\kappa(\tilde \gamma(t))=\gamma(t)$, $t\in \R$.

Next we recall that the deck transforms are the set of diffeomorphisms $\alpha: \Rd\to \Rd$ for which
$$\kappa \circ \alpha = \kappa.$$
The collection of these maps form a group $\Gamma$.  Since $\alpha^*\tg = \tg$, $\alpha \in \Gamma$, any deck transform preserves angles and
distances.  Consequently, the image of any geodesic in $(\Rd,\tg)$ under a deck transform is also a geodesic in this space.  As a result, the collection
of all $\alpha \in \Gamma$ for which $\alpha(\tilde \gamma)=\tilde \gamma$ is a subgroup of $\Gamma$, which is called the stabilizer subgroup
of $\tilde \gamma$ that we denote by $\Stab$.  If $\{\gamma(t): \, t\in \R\}$ is not a periodic geodesic, i.e., if there is no $t_0>0$ so that
$\gamma(t+t_0)=\gamma(t)$, $t\in \R$, then $\Stab$ is just the identity element in $\Gamma$.  If the extension of $\gamma\in \varPi$ is periodic with
minimal period $t_0>0$ then $\Stab$ is a cyclic subgroup which we can write as $\{\alpha_\ell: \, \ell \in {\mathbb Z}\}$, where
$\alpha_\ell$ is determined by $\alpha_\ell(\tilde \gamma(t))=\tilde \gamma(t+\ell t_0)$, $\ell =0, \pm 1, \pm 2, \dots$.  Thus, restricted to $\tilde \gamma$,
$\alpha_\ell$ just involves shifting the geodesic $\tilde \gamma(t)$ by $\ell$ times its period, and $\Stab$ is generated by $\alpha_1$.

Next, let
$$D_{Dir}=\{ \tilde y\in \Rd: \, d_{\tg}(0,\tilde y)<\dgt(0, \alpha(\tilde y)), \, \forall \alpha \in \Gamma, \, \, \alpha \ne Identity\}$$
be the Dirichlet domain for $(\Rd,\tg)$.  We can then add to $D_{Dir}$ a subset of $\partial D_{Dir}=\overline{D_{Dir}}\backslash
\text{Int }(D_{Dir})$ to obtain a natural fundamental domain $D$, which has the property that $\Rd$ is the disjoint union of the
$\alpha(D)$ as $\alpha$ ranges over $\Gamma$ and $\{\tilde y\in \Rd: \, \dgt(0,\tilde y)<10\} \subset D$ since we are assuming that the injectivity
radius of $(M,g)$ is more than ten.

Given $x\in M$, let $\tilde x\in D$ be the unique point in our fundamental domain for which $\kappa( \tilde x)=x$.  We then have (see e.g.  \cite[\S 3.6]{soggehang}) that the kernel of $\cos(\tau\sqrtl)$ can be written as
\begin{equation}\label{p.12}
\bigl(\cos(\tau\sqrtl)\bigr)(x,y)=\sum_{\alpha \in \Gamma}\bigl(\cos \tau\sqrtg\bigr)(\tilde x, \alpha(\tilde y)),
\end{equation}
where $\cos (\tau\sqrtg ): \, L^2(\Rd,\tilde g)\to L^2(\Rd, \tg)$ is the cosine transform associated with $\tg$.
Thus, if $dV_{\tilde g}$ is the associated volume element, we have that when $f\in C^\infty_0(\Rd)$
$$u(\tau, \tilde x)=\int_{\Rd}\bigl(\cos \tau \sqrtg\bigr)(\tilde x, \tilde z) \, f(\tilde z)\, dV_{\tilde g}(\tilde z)$$
is the solution of the Cauchy problem (with $D_\tau = -i \prtl_t$)
$$(D_\tau^2-\Delta_{\tg}) u=0, \quad u|_{\tau=0}=f, \quad \partial_\tau u|_{\tau=0}=0.$$
Therefore, by Huygens principle,
\begin{equation}\label{p.13}
\bigl(\cos\tau\sqrtg\bigr)(\tilde x,\tilde z)=0 \quad \text{if } \, \, \dgt(\tilde x,\tilde z)>|\tau|.
\end{equation}
Also, this kernel is smooth when $\dgt(\tilde x,\tilde z)\ne |\tau|$, i.e.,
\begin{equation}\label{p.14}
\text{sing supp } \bigl(\cos \tau \sqrtg\bigr)(\, \cdot \, ,\, \cdot \, )
\subset \{(\tilde x,\tilde z)\in \Rd\times \Rd: \, \dgt(\tilde x,\tilde z)=|\tau|\}.
\end{equation}

To proceed, we need a result which follows from the Hadamard parametrix and stationary phase:

\begin{lemma}\label{lemmap}  Let $m$ be as in \eqref{p.3} and \eqref{p.6}, and, as above, assume that $\beta\in C^\infty_0(\R)$ satisfies
$\beta(\tau)=1$, $|\tau|<\frac32$ and $\beta(\tau)=0$, $|\tau|\ge 2$.  Then if $\la, T\ge1$ and $\tilde x, \tilde y\in \Rd$, we have
\begin{multline}\label{p.15}
\frac1{2\pi T}\int_{-T}^T (1-\beta(\tau))\, \Hat m(\tau/T) e^{i\la \tau} \, \bigl(\cos \tau \sqrtg \bigr)(\tilde x, \tilde y) \, d\tau
\\
=\rho(x,y)\frac{\la^{\frac{d-1}2}}{T} \sum_\pm a_\pm(\la,T; \dgt(\tilde x,\tilde y)) \, e^{\pm i\la \dgt(\tilde x, \tilde y)} \, + \, R(\la,T, \tilde x,\tilde y),
\end{multline}
where
$\rho \in L^\infty(\Rd \times \Rd)\cap C^\infty(\Rd \times \Rd)$,
\begin{equation}\label{p.16}
a_\pm(\la,T;r)=0, \, r\notin [1,T], \quad \bigl| \partial^j_r a_\pm(\la,T;r)\bigr|\le C_j r^{-\frac{d-1}2-j},
\end{equation}
with constants $C_j$ independent of $T,\la\ge 1$, and
\begin{multline}\label{p.17}
R(\la,T; \tilde x,\tilde y)=0 \, \, \text{if } \, \dgt(\tilde x,\tilde y)>T,
\\
 \text{and }\,
|R(\la,T; \tilde x,\tilde y)|\le C_{T,K}\la^{-2-\frac{d-1}2}, \, \, \text{if } \, \tilde x, \tilde y
\in K\Subset \Rd.
\end{multline}
\end{lemma}

This lemma is standard and can essentially be found in \cite{ChenS}, \cite{SZ4} or \cite[\S 3.6]{soggehang}.  So let us postpone its proof to the end of the section and focus now on using it to help us  to prove \eqref{p.9}.

If we combine \eqref{p.10} and \eqref{p.12}, we can write the kernel of our operator as
\begin{equation}\label{p.18}
S_\la(x,y)=\frac1{2\pi T}\sum_{\alpha\in \Gamma} \int_{-T}^T (1-\beta(\tau))\, \Hat m(\tau/T) \, e^{i\la \tau}\,
\bigl(\cos \tau \sqrtg\bigr)(\tilde x, \alpha(\tilde y))\, d\tau,
\end{equation}
with, as in \eqref{p.12}, $\tilde x$ and $\tilde y$ being the unique points in our fundamental domain having the property that
$x=\kappa(\tilde x)$ and $y=\kappa(\tilde y)$, respectively.

In view of \eqref{p.13} the number of nonzero summands in \eqref{p.18} is finite, but, if the sectional curvatures of $(M,g)$ are strictly negative,
the number of such terms grows exponentially with $T$.  Therefore, as in \cite{SZ4} and \cite{ChenS}, it is convenient and natural to split the
sum into the terms in the stabilizer group for $\tilde \gamma$ and everything else.  So let us write
\begin{equation}\label{p.19}
S_\la(x,y)=\Sstab(x,y)+\Sosc(x,y),
\end{equation}
where
\begin{equation}\label{p.20}
\Sstab(x,y)=\frac1{2\pi T}\sum_{\alpha\in \Stab} \int_{-T}^T (1-\beta(\tau))\, \Hat m(\tau/T) \, e^{i\la \tau}\,
\bigl(\cos \tau \sqrtg\bigr)(\tilde x, \alpha(\tilde y))\, d\tau,
\end{equation}
and
\begin{equation}\label{p.21}
\Sosc(x,y)=\frac1{2\pi T}\sum_{\alpha\in \Gamma\backslash \Stab} \int_{-T}^T (1-\beta(\tau))\, \Hat m(\tau/T) \, e^{i\la \tau}\,
\bigl(\cos \tau \sqrtg\bigr)(\tilde x, \alpha(\tilde y))\, d\tau,
\end{equation}
We shall also call the operator associated with the second term in the right side of \eqref{p.19} $\Sosc$ since we shall be able to use oscillatory integral
operator bounds to control it.

The other piece is very easy to estimate.  We claim that
\begin{equation}\label{p.22}
\|\Sstab h\|_{L^2(\tubel)}\le \Bigl(CT^{-\frac12}+C_T\la^{-2}\Bigr)\|h\|_{L^2}, \quad \text{if supp }h\subset \tubel.
\end{equation}
By Young's inequality, this would be a consequence of the following estimate for the kernel
\begin{equation}\label{p.23}
|\Sstab(x,y)|\le CT^{-\frac12}\la^{\frac{d-1}2}+C_T\la^{-2+\frac{d-1}2},
\end{equation}
since we may restrict to $(x,y) \in \tubel \times \tubel$. If  our $\gamma\in \varPi$ is not a segment of a periodic geodesic in $(M,g)$ then $\Stab$ is just the identity element, in which case \eqref{p.23}
follows trivially from Lemma~\ref{lemmap}.  Otherwise, if the geodesic has period $t_0>0$ then as noted before $\Stab=\{\alpha_\ell\}_{\ell \in {\mathbb Z}}$
where $\alpha_\ell(\tilde \gamma(t))=\tilde \gamma(t+\ell t_0)$.  Since $\dgt(\alpha(\tilde w),\alpha(\tilde z))$ is uniformly bounded as $\tilde w$ and $\tilde z$ range over
$D$ and  $\alpha$ over $ \Gamma$, Lemma~\ref{lemmap} also yields, in this case,
\begin{equation}\label{p.24}
|\Sstab(x,y)|\le CT^{-1}\sum_{1\le \ell t_0\le 2T} \la^{\frac{d-1}2}(1+\ell)^{-\frac{d-1}2}+C_T\la^{\frac{d-1}2-2},
\end{equation}
using \eqref{p.16} (with $j=0$) to obtain the first term in the right and \eqref{p.17} to obtain the other term.  Since $d\ge 2$, \eqref{p.24} implies
\eqref{p.23}.  For later use, note that, since the period $t_0$ must be larger than $10$, in view of our assumption on the injectivity radius of
$(M,g)$, the constants in \eqref{p.22} can be chosen to be independent of $\gamma\in \varPi$.

In view of \eqref{p.22}, the proof of \eqref{p.9} would be complete if we could show that
\begin{equation}\label{p.25}
\|\Sosc h\|_{L^2(\tubel)}\le C_T\la^{-\frac14}\|h\|_{L^2}, \quad \text{if supp }h\subset \tubel.
\end{equation}
By Lemma~\ref{lemmap},
\begin{multline}\label{p.26}
\Sosc(x,y)=\rho(x,y)\frac{{\la^{\frac{d-1}2}}}{T} \sum_{\alpha \in \notstab} a_\pm(\la,T;\dgt(\tilde x, \alpha(\tilde y))) \, e^{\pm i\la \dgt(\tilde x,\alpha(\tilde y))}
\\
+R_\la(x,y),
\end{multline}
where, with bounds independent of $\gamma\in \varPi$,
\begin{equation}\label{p.27}
|R_\la(x,y)|\le C_T\la^{-2+\frac{d-1}2}.
\end{equation}
By invoking Young's inequality one more time, we find that by \eqref{p.26} and \eqref{p.27} we would have \eqref{p.25} if we could show that
\begin{multline}\label{p.28}
\Bigl(\int_{\tubel}\Bigl|\int_{\tubel}\rho(x,y)a_\pm(\la,T; \dgt(x,\alpha(y)))e^{\pm i\la \dgt(x,\alpha(y))} \, h(y)\, dy\Bigr|^2 \, dx\Bigr)^{\frac12}
\\
\le C_\alpha \la^{-\frac{d-1}2-\frac14}\|h\|_{L^2}, \quad \alpha \in \notstab.
\end{multline}
Here, to simplify the notation to follow, as we may, we are identifying $\tubel$ with its preimage in $D$ via $\kappa$.  So we have lifted our
calculation to $\Rd$, and $dy$ denotes the volume element coming from the metric $\tg$.

To prove this we shall use the following result which is an immediate consequence of H\"ormander's $L^2$-oscillatory integral theorem in
\cite{HorOsc} (see also   \cite[Theorem 2.1.1]{soggefica}).

\begin{lemma}\label{lemmao}  Let
$$\phi(z;x,y)\in C^\infty({\mathbb R}^m\times {\mathbb R}^{d-1}\times \Rd)$$ be real and
$$a(z;x,y)\in C_0^\infty({\mathbb R}^m\times {\mathbb R}^{d-1}\times \Rd).$$  Assume that the mixed Hessian in the $(x,y)$ variables of
$\phi$ satisfies
$$\text{Rank}\, \Bigl(\frac{\partial^2}{\partial x_j \partial y_k} \phi(z;x,y)\Bigr)\equiv d-1 \quad \text{on supp }a.$$
Then there is a uniform constant $C$ so that for $\la \ge 1$
$$\Bigl(\int_{{\mathbb R}^{d-1}}\Bigl|\int_{\Rd}e^{i\la \phi(z;x,y)}a(z;x,y) f(y)\, dy\Bigr|^2\, dx\Bigr)^{\frac12}
\le C\la^{-\frac{d-1}2}\|f\|_{L^2(\Rd)},$$
where all the integrals are taken with respect to Lebesgue measure.
\end{lemma}

We also require the following simple geometric lemma so that we can use Lemma~\ref{lemmao} to exploit the fact that our tubes are
only have width $\la^{-\frac12}$ to obtain \eqref{p.28}.

\begin{lemma}\label{lemmag}  Suppose that $\alpha \in \notstab$ and that $x_0, y_0\in \tilde \gamma \cap D$.  Then either
$\alpha(y_0)\notin \tilde \gamma$ or $\alpha^{-1}(x_0)\notin \tilde \gamma$ or both.
\end{lemma}

\begin{proof}[Proof of Lemma~\ref{lemmag}]  Since $\alpha\in\notstab$, it follows that $\tilde \gamma$ and $\alpha(\tilde \gamma)$
are distinct or intersect at a unique point $P=P(\tilde \gamma,\alpha)$ (by the Cartan-Hadamard theorem).  In the first case both
$\alpha(y_0)\notin \tilde \gamma$ and $\alpha^{-1}(x_0)\notin \tilde \gamma$.  We also have the desired conclusion if $P\ne \alpha(y_0)$,
for then we must have $\alpha(y_0)\notin \tilde \gamma$ as $\alpha(y_0)\in \alpha(\tilde \gamma)$.

Suppose that we are in the remaining case where $\tilde \gamma \cap \alpha(\tilde \gamma)=\{\alpha(y_0)\}$.  Since $x_0, y_0\in D$ and
$D\cap \alpha(D)=\emptyset$, it follows that $x_0\ne \alpha(y_0)$.  Therefore, as $x_0\in \tilde \gamma$, we must have that $x_0\notin \alpha(\tilde \gamma)$.  Thus, in this case, we must have $\alpha^{-1}(x_0)\notin \tilde \gamma$, meaning that we have the desired
conclusion for this case as well.
\end{proof}

To use these two lemmas we require some simple facts about the Riemannian distance function $\dgt(x,z)$.  We recall that $(\Rd, \tilde g)$ has no
conjugate points.  Thus, the $d\times d$ Hessian $\frac{\partial^2}{\partial x_j\partial z_k}\dgt(x,z)$ has rank identically equal to $d-1$ away from the diagonal.

With this in mind, let us fix points $x_0$ and $y_0$ on our unit geodesic segment $\gamma\subset D$.  We shall now prove a local version of
our remaining estimate \eqref{p.28}.  By Lemma~\ref{lemmag}, for our given $\alpha\in \notstab$, we know that either $\alpha(y_0)\notin \tilde \gamma$
or $\alpha^{-1}(x_0)\notin \tilde \gamma$.  For the moment, let us assume the former, i.e.,
\begin{equation}\label{assumption1}
\alpha(y_0)\notin \tilde \gamma.
\end{equation}

We then have that the geodesic passing through
$z_0=\alpha(y_0)$ and $x_0\in \gamma\subset \tilde \gamma$ intersects $\gamma$ transversally.  We may therefore choose geodesic normal
coordinates in $\Rd$ vanishing at $x_0$ so that $\tilde \gamma$ is the first coordinate axis, i.e.
$$\tilde \gamma=\{(t,0,\dots, 0): \, t\in \R\},$$
and, moreover, if $x'=(x_1,\dots,x_{d-1})$ are the first $d-1$ coordinates of $x$ in this coordinate system then
$$\text{Rank } \, \Bigl( \frac{\partial^2 }{\partial x'_j\partial z_k} \dgt \bigl((x',0),z) \Bigr)\, = \, d-1 \quad \text{at } \, x'=0 \, \, \text{and } \, z=z_0=\alpha(y_0).$$
By Gauss' lemma this will be the case if the geodesic through the origin and $z_0$ intersects the hyperplane $\{x: \, x_n=0\}$ transversally as shown in
Figure 1 below, which can be achieved after performing a rotation fixing the first coordinate axis if needed.  Since $\alpha: \, \Rd\to \Rd$ is a diffeomorphism
it follows that in given our fixed points $x_0, y_0\in \gamma\subset \tilde \gamma \cap D$, we can find $\delta>0$ so that, in the above coordinates,
\begin{equation}\label{p.29}
\text{Rank } \, \Bigl( \frac{\partial^2 }{\partial x'_j\partial y_k} \dgt \bigl((x',x_n),\alpha(y))\Bigr) \, = \, d-1, \quad \text{if } \,
x\in B_\delta(x_0) \, \, \text{and } \, y_0\in B_\delta(y_0),
\end{equation}
with $B_\delta(w)$ denoting the geodesic ball of radius $\delta$ about $x\in \Rd$.

\begin{figure}
\begin{center}
\resizebox{3.5in}{2in}{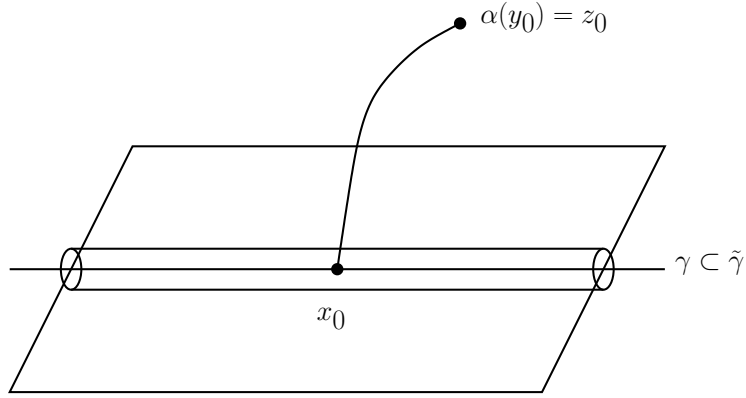}
\end{center}
\caption{Transversal intersections}
\end{figure}

Next, it follows from \eqref{p.26} and Lemma~\ref{lemmao} that, in our coordinates, for each fixed value of $x_n$, we have
\begin{multline*}
\Bigl(\int_{\{x': (x',x_n)\in\tubel\cap B_\delta(x_0)\}}\Bigl| \int_{ \tubel \cap B_\delta(y_0)}
\rho(x, \alpha(y))\, a_\pm (\la,T; \dgt(x,\alpha(y)))
\\ \times
e^{\pm i \la \dgt(x,\alpha(y))} \, h(y)\, dy\, \Bigr|\, dx'\Bigr)^{\frac12}
\le C_\alpha \la^{-\frac{d-1}2}\Bigl(\int |h(y)|^2 \, dy\Bigr)^{\frac12}.
\end{multline*}
Since $|x_n|\lesssim \la^{-\frac12}$ in $\tubel$, from this, we deduce that, under our assumption \eqref{assumption1}, we have that
\begin{multline}\label{balltube}
\Bigl(\int_{\tubel\cap B_\delta(x_0)}\Bigl| \int_{\tubel \cap B_\delta(y_0)}
\rho(x, \alpha(y))\, a_\pm (\la,T; \dgt(x,\alpha(y)))
\\ \times
e^{\pm i \la \dgt(x,\alpha(y))} \, h(y)\, dy\, \Bigr|\, dx\Bigr)^{\frac12}
\le C_\alpha \la^{-\frac{d-1}2} \, \la^{-\frac14} \, \Bigl(\int |h(y)|^2 \, dy\Bigr)^{\frac12}.
\end{multline}

Lemma~\ref{lemmag} tells us that if we do not have \eqref{assumption1}  then
\begin{equation}\label{assumption2}
\alpha^{-1}(x_0)\notin \tilde \gamma.\end{equation}
We claim that for our fixed points $x_0, y_0\in \gamma$ we can find $\delta>0$ so that \eqref{balltube} remains valid for this case as well.  To do this, we just use the fact that our $\alpha\in \notstab$ is an isometry and therefore
$$\dgt(x,\alpha(y))=\dgt(\alpha^{-1}(x),y).$$
Consequently, since $\alpha^{-1}\in \notstab$, we obtain \eqref{balltube} under the assumption \eqref{assumption2} since it is essentially just the dual version of the case we just handled, and so follows from the above argument after taking adjoints.

Since we have shown that \eqref{balltube} holds either under assumption \eqref{assumption1} or \eqref{assumption2}, Lemma~\ref{lemmao} tells us that given any two fixed points $x_0, y_0\in \gamma$ we can find a $\delta>0$ so that \eqref{balltube} is valid.  By the compactness of our unit geodesic
segment $\gamma$, this implies \eqref{p.28}, which completes the proof of the estimate \eqref{p.2} for our fixed $\gamma\in \varPi$.

It is straightforward to see how to obtain the stronger estimate \eqref{p.1}, which involves uniform bounds over $\varPi$, by using the proof of
\eqref{p.2}.  We use the fact that if $T\gg 1$ is fixed and if $\gamma\in \varPi$ is fixed then there is a neighborhood ${\mathcal N}(\gamma)$ of
$\gamma$ in $\varPi$ so that if $\alpha \in \notstab$ and the geodesic distance between our fundamental domain $D$ and its image $\alpha(D)$
is $\le 2T$, then we also have that $\alpha \notin \Gamma\backslash \text{Stab }(\tilde \gamma_0)$ for any $\gamma_0\in {\mathcal N}(\gamma)$.
This follows from the fact that there are only finitely many $\alpha\in \Gamma$ for which the distance between $D$ and $\alpha(D)$ is $\le 2T$, and
if $\alpha$ is not a stabilizer for $\tilde \gamma$ then it is also not a stabilizer for nearby geodesics.

Because of this and the uniform dependence on the smooth parameter $z$ in Lemma~\ref{lemmao}, if we define $\Soscg$ to be the
operator whose kernel is given by \eqref{p.21}, we have the uniform bounds
$$\|\Soscg h\|_{L^2(\tubel)}\le C_T\la^{-\frac14}\|h\|_{L^2}, \quad \text{if } \, \gamma_0\in {\mathcal N}(\gamma) \,
\, \text{and  supp }h\subset \mathcal{T}_{\l^{-1/2}}(\gamma_0).
$$
If then $\Sstabg=S_\la-\Soscg$ is then defined using $\gamma$, then the proof of \eqref{p.22} clearly also yields the following variant
$$\|\Sstabg h\|_{L^2(\tubel)}\le \bigl(CT^{-\frac12}+C_T\la^{-2}\bigr) \,
\|h\|_{L^2}, \, \,  \text{if } \, \gamma_0\in {\mathcal N}(\gamma) \,
\, \text{and  supp }h\subset \mathcal{T}_{\l^{-1/2}}(\gamma_0).
$$
Together  these two estimates imply the analog of \eqref{p.1} where, instead of having the geodesic segments range over $\varPi$, we have them
range over ${\mathcal N}(\gamma)$ and $\Lambda_\e=\Lambda_\e({\mathcal N}(\gamma))$ depends on ${\mathcal N}(\gamma)$.  By the
compactness of $\varPi$, this in turn yields \eqref{p.1}.

\bigskip

To wrap things up, we also need to prove Lemma~\ref{lemmap}.

\begin{proof}[Proof of Lemma~\ref{lemmap}]
Since $\Hat m(\tau)=0$ when $|\tau|>1/2$ it follows that the left side of \eqref{p.15},
\begin{equation}\label{p.30}
\frac1{2\pi T} \int_{-T}^T(1-\beta(\tau))\, \Hat m(\tau/T) \, e^{i\la \tau} \, \bigl(\cos \tau \sqrtg\bigr)(\tilde x, \tilde y)\, d\tau,
\end{equation}
vanishes when $\dgt(\tilde x,\tilde y)>T$.  Since $\beta(\tau)=1$ for $|\tau|\le 3/2$, by \eqref{p.14}, it is $O_{N,T}((1+\la)^{-N})$ for any
$N=1,2,3,\dots$ if $\dgt(\tilde x, \tilde y)\le 1$.  Therefore, we need only to prove the assertions in Lemma~\ref{lemmap} when $1\le \dgt(\tilde x,\tilde y)\le T$.

To prove this, we shall use the Hadamard parametrix (see e.g. \cite{Hor3} and \cite[Chapter 2]{soggehang}).  Since $(\Rd,\tilde g)$ has nonpositive curvature, for
$0\le \tau \le T$ we can write
\begin{align}\label{p.31}
\bigl(\cos \tau \sqrtg\bigr)(\tilde x, \tilde y)&=\rho(\tilde x,\tilde y) \, (2\pi)^{-d}\int_{\Rd}e^{i\dgt(\tilde x, \tilde y)\xi_1}\cos \tau|\xi| \, d\xi
\\
&=\sum_\pm \int_{\Rd}e^{i\dgt(\tilde x,\tilde y)\xi_1} \alpha_\pm(\tau,\tilde x, \tilde y,|\xi|)\, e^{\pm i\tau |\xi|}\, d\xi
\notag
\\
&+R(\tau,\tilde x, \tilde y), \notag
\end{align}
where the leading Hadamard coefficient, $\rho$,  is smooth and uniformly bounded (by the curvature hypothesis), and if $m\in {\mathbb N}$ is fixed we can
have $\partial_\tau^j R(\tau, \tilde x,\tilde y)\in L^\infty_{loc}$, $0\le j\le m$, and also
\begin{multline}\label{p.32}
|\partial^\beta_{\tau,\tilde x, \tilde y}\partial^j_r \alpha_\pm(\tau, \tilde x, \tilde y,r)| \le  C_{T,K,\beta,j} \, r^{-2-j},
\\ \text{if } \, r\ge 1, \, \, 0\le \tau\le T, \, \, j=0, 1, 2,\dots, \, \, \text{and } \, \tilde x, \tilde y\in K\Subset \Rd.
\end{multline}

We also recall (see e.g. \cite{soggefica}) that we can write the Fourier transform of Lebesgue measure on the sphere in $\Rd$ as
\begin{equation}\label{p.33}
\int_{S^{d-1}}e^{ix\cdot \omega}\, d\sigma(\omega) \, = \, |x|^{-\frac{d-1}2}\Bigl(c_+(|x|)e^{i|x|}+c_-(|x|)e^{-i|x|}\Bigr),
\end{equation}
where for each $j=0,1,2,\dots$, we have
\begin{equation}\label{p.34}
|\partial^j_r c_+(r)|+|\partial^j_r c_-(r)|\le C_j r^{-j}, \quad r\ge 1.
\end{equation}

If in \eqref{p.30} we replace $(\cos \tau \sqrtg)(\tilde x, \tilde y)$ by the first term in \eqref{p.31}, the resulting expression equals
$\rho(\tilde x, \tilde y)$ times a fixed multiple of
\begin{multline}\label{p.35}
\frac1{2\pi T} \int_{-T}^T \int_{\Rd} \Hat m(\tau/T) \, e^{i\la \tau} \, \cos (\tau|\xi|) \, e^{i\dgt(\tilde x, \tilde y)\xi_1} \, d\xi d\tau
\\
=\sum_{\pm}\frac1{2\pi T }\int_{-T}^T\int_0^\infty \Hat m(\tau/T) \, e^{i\la \tau} \, \cos (\tau r) e^{\pm ir \dgt(\tilde x, \tilde y)}
c_\pm\bigl(\dgt(\tilde x, \tilde y)r\bigr) \, \frac{r^{\frac{d-1}2}}{(\dgt(\tilde x,\tilde y))^{\frac{d-1}2}} \, dr d\tau
\end{multline}
minus
\begin{equation}\label{p.36}
\sum_\pm \frac1{2\pi T}
\int_{-2}^2\int_0^\infty
\beta(\tau) \,
 \Hat m(\tau/T) \, e^{i\la \tau} \, \cos (\tau r) e^{\pm ir \dgt(\tilde x, \tilde y)}
c_\pm\bigl(\dgt(\tilde x, \tilde y)r\bigr) \, \frac{r^{\frac{d-1}2}}{(\dgt(\tilde x,\tilde y))^{\frac{d-1}2}} \, dr d\tau.
\end{equation}
If we replace $\cos(\tau r)$ by $e^{-i\tau r}$ in the right side of \eqref{p.35}, the resulting expression equals  the sum over $\pm$ of
\begin{multline*}\int_0^\infty m\bigl(T(\la-r)\bigr) \, c_\pm \bigl(\dgt(\tilde x, \tilde y) r\bigr) e^{\pm ir\dgt(\tilde x, \tilde y)}\, \frac{r^{\frac{d-1}2}}{(\dgt(\tilde x,\tilde y))^{\frac{d-1}2}}  \, dr
\\
=\frac{\la^{\frac{d-1}2}}{T} e^{\pm i\la \dgt(\tilde x, \tilde y)} a_\pm(\la,T; \dgt(\tilde x, \tilde y)),
\end{multline*} 
where, using the fact that $m\in {\mathcal S}(\R)$ and \eqref{p.34}, $a_\pm$ satisfies \eqref{p.16}.  If in \eqref{p.35} we replace
$\cos(\tau r)$ by $e^{i\tau r}$, then this argument also implies that the resulting expression is $O_{N,T}((1+\la)^{-N})$ for any $N=1,2,3,\dots$.
Thus, modulo such an error $\rho$ times the terms in \eqref{p.35} can be written as the first term in the right side of \eqref{p.15} with
\eqref{p.16} being valid.  Since this argument shows that the same is the case for \eqref{p.36}, we conclude that the first term in the right
side of \eqref{p.31}, up to  $O_{N,T}((1+\la)^{-N})$ errors, gives us the first term in the right side of \eqref{p.15}.

This argument  and \eqref{p.32} also shows that if in \eqref{p.15} we replace $(\cos \tau \sqrtg)(\tilde x, \tilde y)$ by the second
term in the right side of \eqref{p.31}, then we get a term obeying the bounds in \eqref{p.17}.  Since, as noted we can take the remainder
term in \eqref{p.31} to satisfy for a given $m\in {\mathbb N}$, $\partial^j_\tau R(\tau, \tilde x, \tilde y)\in L^\infty_\loc$, $j=0,1,\dots, m$, we also see that if we choose $m$ large enough, the same is true for it.
\end{proof}

\end{document}